\documentclass[a4paper, 11pt]{article}
\usepackage{amsmath,amssymb,color}
\usepackage[dvips]{graphicx}
\usepackage{verbatim}
\usepackage{amsthm}
\usepackage{amsmath,amssymb,color}
\usepackage{color} 
\RequirePackage[OT1]{fontenc}
\RequirePackage{amsthm,amsmath, amssymb, enumerate}
\RequirePackage[round, sort&compress, authoryear]{natbib}
\RequirePackage[colorlinks,citecolor=blue,urlcolor=blue]{hyperref}

\usepackage{eufrak}
\usepackage{graphicx, amsthm, url,pstricks,fancyhdr}
\RequirePackage{bbm,amsthm}
\RequirePackage{amssymb}

\RequirePackage{enumitem}
\RequirePackage{bigints}
\RequirePackage{mathtools}
\RequirePackage[stable]{footmisc}
\usepackage{bm}

\theoremstyle{plain}
\newtheorem{theorem}{Theorem}[section]                                          
\newtheorem{proposition}[theorem]{Proposition}                          

\newtheorem{corollary}[theorem]{Corollary}

\theoremstyle{definition}

\theoremstyle{remark}
\newtheorem{remark}[theorem]{Remark}

\makeatletter \@addtoreset{equation}{section} \makeatother

\newcommand{\Prob}{\mathbb{P}\,}
\newcommand{\gr}[1]{{\color{black} #1}}

\newcommand{\C}{\mathbb{C}}

\newcommand{\bea}{\begin{eqnarray}}
\newcommand{\ena}{\end{eqnarray}}
\newcommand{\beq}{\begin{equation}}
\newcommand{\enq}{\end{equation}}
\newcommand{\beas}{\begin{eqnarray*}}
\newcommand{\enas}{\end{eqnarray*}}

\newcommand{\PP}{{\mathbb{P}} }

\newcommand{\bX}{{\bf{X}}}

\newcommand{\EE}{{\mathbb{E}} }
\newcommand{\RR}{{\mathbb{R}} }

\def\heiko{\mathrel{\raise.3ex\hbox{\scalebox{.7}{%
    \rotatebox[origin=c]{-7}{/}%
    \kern-.35em\rotatebox[origin=c]{-7}{/}}}}}%

\title{Bounds for the asymptotic distribution of the likelihood ratio} 
\author{Andreas Anastasiou$^{1}$ and Gesine Reinert$^{2}$ 
\\
\small{$^{1}$ The London School of Economics and Political Science}
\\
\small{$^{2}$ University of Oxford}\\
}
\date{\today}
\begin{document}
\pagenumbering{roman}
\pagenumbering{arabic}
\maketitle

\begin{abstract}
\indent In this paper we give an explicit bound on the distance to chisquare for the likelihood ratio statistic when the data are realisations of independent and identically distributed random elements. To our knowledge this is the first explicit bound which is available in the literature. The bound depends on the number of samples as well as on the dimension of the parameter space. We illustrate the bound with three examples: samples from an exponential distribution, samples from a normal distribution, and logistic regression. \end{abstract}


\section{Introduction}

One of the most celebrated theorems in theoretical statistics is Wilks' Theorem, which states that under appropriate conditions, $2 \times $ a log-likelihood ratio statistic is approximately chisquare distributed. This result is very useful when testing the null hypothesis $H_0: \boldsymbol{\theta} \in \Theta_0$ against the alternative hypothesis $H_1: \boldsymbol{\theta} \in \Theta$, where $\Theta_0 \subset \Theta$ using a generalised likelihood ratio test. Such tests arise for example in the area of model reduction, with the aim of finding a relatively simple model which explains the data reasonably well, see for example Chapter 6.5 in \cite{cox2006principles}.  For this test the degrees of freedom of the asymptotic chisquare distribution under the null hypothesis is $df = {\rm dim}(\Theta) - {\rm dim}(\Theta_0)$;  see \cite{wilks1938large}, as well as \cite{lehmann2006testing} and \cite{van1998asymptotic} for more details. 

Here is a sketch of the argument. Let $\boldsymbol{X} = ( \boldsymbol{X_1}, \ldots, \boldsymbol{X_n})$ be independent and identically distributed (i.i.d.) observations from a distribution with probability density function $f(\boldsymbol{x}| \boldsymbol{\theta})$, where $\boldsymbol{\theta} {\gr{= (\theta_1, \ldots, \theta_d)^\intercal }} \in \Theta \subset \RR ^d$. Note that we do not make any assumptions on the space where the observations live. The test problem is 
$$H_0: \theta_{0,j} = 0, \quad j=1, \ldots, r$$
against the general alternative $H_1: \boldsymbol{\theta} \in \Theta$. Assume that ${\rm dim}(\Theta) = d$; then $\Theta_0=\{\boldsymbol{\theta} \in \Theta: \theta_{0,j} = 0 \mbox{ for } j=1, \ldots, r\}$ has dimension $d-r$. Writing 
$
\boldsymbol{\theta} = \left(\boldsymbol{\theta}_{[1:r]}, \boldsymbol{\theta}_{[r+1:d]}\right)^\intercal 
$
where $\boldsymbol{\theta}_{[1:r]}$ is the  vector of the first $r$ components of $\boldsymbol{\theta}$ and $\boldsymbol{\theta}_{[r+1:d]}$ is the vector of the remaining $d-r$ components of $\boldsymbol{\theta}$, 
the null hypothesis translates to $H_0: \boldsymbol{\theta}_{0, [1:r]} =\boldsymbol{0}$.

Let $L(\boldsymbol{\theta}; \boldsymbol{x}) = \prod_{i=1}^n f(\boldsymbol{x_i}|\boldsymbol{\theta})$ denote the likelihood function which is assumed to be regular, so that the maximum likelihood estimate exists and is unique, and having derivatives of up to third order with respect to $\boldsymbol{\theta}$. The conditions will be made precise in Section \ref{multi}. Set 
\begin{align}
\nonumber & \boldsymbol{\hat{\theta}^{{\rm res}}(x)} = {\rm argmax}_{\boldsymbol{\theta} \in \Theta_0} L(\boldsymbol{\theta};\boldsymbol{x})  = \left(\boldsymbol{0_}{[1:r]}, \boldsymbol{\hat{\theta}^{*}_{[r+1:d]}(x)}\right)^\intercal\\
\nonumber & \boldsymbol{\hat{\theta}_n(x)} = {\rm argmax}_{\boldsymbol{\theta} \in \Theta} L(\boldsymbol{\theta}; \boldsymbol{x}).
\end{align}
In the sequel we often abbreviate $\boldsymbol{\hat{\theta}^{*}(x)}= \boldsymbol{\hat{\theta}^{*}_{[r+1:d]}(x)}$. 
The log-likelihood ratio statistic is 
\begin{equation}
\label{loglikelihoodsplitT1T2}
- 2 \log \Lambda = 2 \log \left( \frac{T_1}{T_2}  \right) 
\mbox{ with }  T_1 = \frac{ L(\boldsymbol{\hat{\theta}_n(x)};\boldsymbol{x})}{L(\boldsymbol{\theta_0};\boldsymbol{x})}  \mbox{ and } T_2 = \frac{ L(\boldsymbol{\hat{\theta}^{res}(x)};\boldsymbol{x})}{ L(\boldsymbol{\theta_0};\boldsymbol{x})}
\end{equation}
with $\boldsymbol{\theta_0}$ the unknown true parameter. Thus, $T_1$ is the likelihood ratio for testing the simple null hypothesis that $\boldsymbol{\theta} = \boldsymbol{\theta_0}$ against the alternative that $\boldsymbol{\theta} \in \Theta$, whereas $T_2$ is the likelihood ratio for testing the simple null hypothesis that $\boldsymbol{\theta} = \boldsymbol{\theta_0}$ against the alternative that $\boldsymbol{\theta} \in \Theta_0$. The Fisher information matrix for one random vector is denoted by $I(\boldsymbol{\theta_0})$,
which again is assumed to exist. 
We write 
$\ell(\boldsymbol{\theta}; \boldsymbol{x}) = \log (L(\boldsymbol{\theta};\boldsymbol{x})) = \sum_{i=1}^n \log (f(\boldsymbol{x_i}| \boldsymbol{\theta}))$ and we also abbreviate
$\ell_{\bf{x_i}}(\boldsymbol{\theta}) = \log (f(\boldsymbol{x_i}| \boldsymbol{\theta})).$ In addition, let ${\bf{Y}}_j = \nabla \ell_{\bf{X_j}}(\theta) = (Y_{i,j}, j=1, \ldots, d)$ with $\nabla = (\frac{\partial}{\partial \theta_i}, i=1, \ldots, d)^{\intercal}$ the gradient operator, so that $Y_{i,j} = \frac{\partial}{\partial \theta_i} \log (f(\boldsymbol{X_j}
| \boldsymbol{\theta}))$. A  key observation is that
 ${\bf{Y}}_j , j=1, \ldots, n$ are i.i.d. vectors of possibly dependent entries.
Under the regularity assumptions in this paper, for every fixed $i$ the sum $\sum_{j=1}^n Y_{i,j}$ satisfies a law of large numbers as well as a central limit theorem.

As $\frac{\partial}{\partial \theta_i} \ell\left(\boldsymbol{\hat{\theta}_n(x)};\boldsymbol{x}\right) = 0$, for $i=1, 2, \ldots, d$, a Taylor expansion gives 
\bea 
2 \log T_1
 &\approx&   n \left(\boldsymbol{\hat{\theta}_n(x)}- \boldsymbol{\theta_0}\right)^{\intercal} I(\boldsymbol{\theta_0}) \left(\boldsymbol{\hat{\theta}_n(x)} - \boldsymbol{\theta_0}\right).\label{t1exp0} 
\ena
A similar approximate expression for $2 \log T_2$ holds,  but with $\boldsymbol{\hat{\theta}_n(x)}$ replaced by $\boldsymbol{\hat{\theta}^{{res}}(x)}$. 
The score function for $\boldsymbol{\theta_0}$ is
\begin{eqnarray} \label{score} S(\boldsymbol{\theta_0}) = S(\boldsymbol{\theta_0},\boldsymbol{x}) = \nabla \log L(\boldsymbol{\theta_0};\boldsymbol{x}) = \sqrt{n}
 \left( \begin{array}{c  }\boldsymbol{\xi} (\boldsymbol{\theta_0}, \boldsymbol{x})   \\   \boldsymbol{\eta} (\boldsymbol{\theta_0}, \boldsymbol{x}) \end{array} \right)  
\end{eqnarray}
with column vectors  $\boldsymbol{\xi} = (\xi_1, \ldots, \xi_r)^{\intercal} \in \RR^r$ and $\boldsymbol{\eta}= (\eta_1, \ldots, \eta_{d-r})^{\intercal} \in \RR^{d-r}$. We omit the arguments $\boldsymbol{x}$ and $\boldsymbol{\theta}$ when they are obvious from the context. Under the regularity conditions in this paper, the $(d \times 1)$ score function  and the $d \times d$ Fisher information matrix are connected through the equality  
\begin{eqnarray} \label{Fisherinfo} \EE [S(\boldsymbol{\theta_0}) S(\boldsymbol{\theta_0})^{\intercal} ] = nI(\boldsymbol{\theta_0}).
\end{eqnarray} 
By Taylor expansion, 
\begin{equation}
\nonumber \sqrt{n}\left( \begin{array}{c  }\boldsymbol{\xi} (\boldsymbol{\theta_0}, \boldsymbol{x})   \\   \boldsymbol{\eta} (\boldsymbol{\theta_0}, \boldsymbol{x}) \end{array} \right)  \approx  n I(\boldsymbol{\theta_0}) (\boldsymbol{\hat{\theta}_n(x)} - \boldsymbol{\theta_0}). 
\end{equation}  
Re-arranging  
and \eqref{t1exp0} yields
\bea \label{t1exp20} 
2 \log T_1
&\approx&\left( \begin{array}{c  }\boldsymbol{\xi}  \\   \boldsymbol{\eta} \end{array} \right)  ^{\intercal} [I(\boldsymbol{\theta_0})]^{-1}
\left( \begin{array}{c  }\boldsymbol{\xi}   \\   \boldsymbol{\eta}\end{array} \right)  
  .
\ena
Similarly to \eqref{t1exp20} we obtain an expression for $2 \log T_2$. Because $\boldsymbol{\xi}=\boldsymbol{0}$ under $H_0$, then only the lower right corner, denoted by $C$, of the expected Fisher information matrix plays a role, so that under $H_0$, 
\bea \label{t2exp0} 
2 \log T_2
&\approx&  \boldsymbol{\eta}^{\intercal} C^{-1} \boldsymbol{\eta} .
\ena
 For the likelihood ratio statistic, \eqref{t1exp20} and \eqref{t2exp0} give that under $H_0$, 
\bea  \label{approxsymm}
- 2 \log \Lambda \approx 
\left( \begin{array}{c  }\boldsymbol{\xi}  \\   \boldsymbol{\eta} \end{array} \right)  ^{\intercal} [I(\boldsymbol{\theta_0})]^{-1}
\left( \begin{array}{c  }\boldsymbol{\xi}   \\   \boldsymbol{\eta}\end{array} \right)  
- \boldsymbol{\eta}^{\intercal} C^{-1} \boldsymbol{\eta}.
\ena  
This difference 
is an even function in $(\boldsymbol{\xi}, \boldsymbol{\eta})$. Hence, chisquare approximation results from \cite{Gaunt_Reinert} for symmetric test functions apply to $(\boldsymbol{\xi}, \boldsymbol{\eta})$. These results lead to an overall bound to the chisquare distribution with $r$ degrees of freedom. 

For any such generalised likelihood ratio test there are only finitely many observations available, and the quality of the approximation will depend on the number of observations, and also on the distribution of the observations under the null hypothesis. To date, bounds on the distance to the chisquare distribution are only available in special cases. This paper addresses the problem through the use of Stein's method. The key ingredients are \cite{anastasiou2015assessing}, where the distance to normality for maximum likelihood estimators is bounded using Stein's method for multivariate normal approximation, and \cite{Gauntetal}, where Stein's method for chisquare approximation is developed. We shall apply these results in order to obtain our main theorem, Theorem \ref{Theorem_i.n.i.d}. This theorem gives an explicit bound on the distance between the log-likelihood ratio statistic and the corresponding chisquare distribution. 

In some instances the bounds will not be small. Indeed likelihood ratio tests are not generally asymptotically chisquare distributed when the number of parameters is not negligible compared to the number of observations, see for example \cite{sur2017likelihood}, and hence the bounds should not always be small. 

Our results are the first ones which give an explicit bound to the chisquare distribution in Wilks' theorem under a general setting.  Our bounds in this paper are not optimised with respect to the constants. Their importance is of theoretical nature, but they can also be viewed as indicative of situations when the chisquare approximation does not hold. To illustrate this point, if $d$ is the dimension of the parameter space and $n$ the number of observations, then our bounds are of order $d^{7}/\sqrt{n}$. Hence the dimension of the parameter space is allowed to increase with $n$, but only very slowly. In particular in  the regime considered in \cite{sur2017likelihood}, that $d$ grows linearly with $n$, our bounds will tend to infinity with increasing sample size, as they should in this case. 

In \cite{Gauntetal}, for the Pearson chisquare statistic with fixed number of cells, a bound to the chisquare distribution of order $n^{-1}$ is obtained, through making use of the quadratic form of the chisquare statistic. In contrast, Theorem \ref{Theorem_i.n.i.d} gives a bound of the order $n^{-1/2}$ for fixed $d$, with no clear possibility of improving the bounds. While the proof of Theorem \ref{Theorem_i.n.i.d}  approximates the log-likelihood ratio by a quadratic form, this approximation is bounded by a quantity of order $n^{-1/2}$ . It has been suggested in the past that Pearson's chisquare statistic is closer to a chisquare distribution than the corresponding log-likelihood ratio statistic, see for  example the chapter on historical perspective in \cite{read2012goodness}. It is probable that including a Bartlett correction in the log-likelihood ratio statistic as in \cite{wit2012borrowing} will improve its asymptotic performance, see Chapter 6.11 in  \cite{cox2006principles}. In future work it will be interesting to explore the discrepancy between the two tests further. 

In this paper the observations $\bf{X}$ are assumed to be independent and identically distributed. As our proof is mainly based on Stein's method, generalisations to weakly dependent observations are straightforward in principle, see for example \cite{chen2010normal} and references therein. The results from \cite{Gaunt_Reinert} which we use are established for independent vectors of possibly dependent observations. For simplicity of exposition this paper concentrates on the classical i.i.d. case.

The conditions in our paper are such that the log-likelihood is locally linear, and hence resembles a quantitative approach to locally asymptotically normal models  in the sense of Le Cam \cite{le1960locally}. In contrast to Le Cam's general theory, instead of considering any small perturbation around the true parameter we restrict attention to the maximum-likelihood estimator. This restriction allows to apply results from \cite{Anastasiou_Reinert}. Expanding the results to provide a quantitative framework for Le Cam's theory will be part of future work. 

The paper is structured as follows. Section \ref{multi} gives the general result. The proof is presented in modular form as a collection of propositions, lemmas and corollaries, because the different steps in the approximation may be of independent interest. In particular, the right results could be adapted to provide the distance to chisquare for the score statistic. Section \ref{sec:examples} illustrates the result in three examples. Firstly we consider an example with a one-dimensional parameter, namely the exponential distribution. The second example is that of the normal distribution with two-dimensional parameter $(\mu, \sigma^2)$. The last example is logistic regression. 


\section{The general result} \label{multi}

We now make the argument from the Introduction precise. Let $\boldsymbol{X} = ( \boldsymbol{X_1}, \boldsymbol{X_2, }\ldots, \boldsymbol{X_n})$ be i.i.d. observations from a distribution with probability density function $f(\boldsymbol{x}| \boldsymbol{\theta})$, where $\boldsymbol{\theta} \in \Theta \subset \RR^d$.  Let $dim(\Theta) = d$ so that  $\Theta_0=\{ \boldsymbol{\theta_0} \in \Theta: \theta_{0,j} = 0 \mbox{ for } j=1, \ldots, r\}$ has dimension $d-r$. The test problem is $H_0: \boldsymbol{\theta_0},_{[1:r]} =0$ against the general alternative. 
Expectations are taken under the true parameter $\boldsymbol{\theta_0}$. 

\subsection{The regularity assumptions} \label{regularity}

In this section we specify the regularity assumptions in the multidimensional case. The assumptions are made to ensure that the maximum likelihood estimator (MLE) for $\boldsymbol{\theta}$ exists, is unique, asymptotically consistent and asymptotically normal. There are different sets of such assumptions available; here we shall use the set-up  for asymptotic normality as in Section 4.4.2 of \cite{Davison}. The subscript $\boldsymbol{\theta}$ in $\EE_{\boldsymbol{\theta}}$ signifies that the expectation is taken under $f(\boldsymbol{x}|\boldsymbol{\theta})$. From now on, we write 
\bea \label{Fisher} I(\boldsymbol{\theta_0}) =  \left( \begin{array}{c c } A & B \\B^{\intercal} & C   \end{array} \right),
\ena  
where for any $r \in \left\lbrace 1,2,\ldots, d\right\rbrace$, $A = A^\intercal \in \mathbb{R}^{r\times r}$, $B \in \mathbb{R}^{r\times(d-r)}$, and $C = C^\intercal  \in \mathbb{R}^{(d-r)\times(d-r)}$. Following \cite{Davison}, we make the following assumptions:
\begin{itemize}[leftmargin=0.55in]
\item[(R.C.1)] The densities defined by any two different values of $\boldsymbol{\theta}$ are distinct;
\item[(R.C.2)] $\ell(\boldsymbol{\theta};\boldsymbol{x})$ is three times differentiable with respect to the unknown vector parameter, $\boldsymbol{\theta}$, and the third partial derivatives are continuous in $\boldsymbol{\theta}$;
\item[(R.C.3)] for any $\boldsymbol{\theta_0} \in \Theta$ and for $\boldsymbol{\mathbb{X}}$ denoting the support of the data, there exists $\epsilon(\boldsymbol{\theta_0}) > 0$ and functions $M_{rst}(\boldsymbol{x})$ (they can depend on $\boldsymbol{\theta_0}$), such that for $\boldsymbol{\theta} = (\theta_1, \theta_2, \ldots, \theta_d)$ and $r, s, t, j = 1,2,\ldots,d,$
\begin{equation}
\nonumber \left|\frac{\partial^3}{\partial \theta_r \partial \theta_s \partial \theta_t}\ell(\boldsymbol{\theta};\boldsymbol{x})\right| \leq M_{rst}(\boldsymbol{x}), \; \forall \boldsymbol{x} \in \boldsymbol{\mathbb{X}},\; \left|\theta_j - \theta_{0,j}\right| < \epsilon(\boldsymbol{\theta_0}),
\end{equation}
with $\EE[M_{rst}(\boldsymbol{X})] < \infty$;
\item[(R.C.4)] for all $\boldsymbol{\theta} \in \Theta$, $\EE_{\boldsymbol{\theta}}[\ell_{\boldsymbol{X_i}}(\boldsymbol{\theta})] = 0$;
\item[(R.C.5)] the expected Fisher information matrix for one random vector, $I(\boldsymbol{\theta})$ is finite, symmetric and positive definite. For $r,s=1,2,\ldots,d$, its elements satisfy
\begin{equation*}
\nonumber n[I(\boldsymbol{\theta})]_{rs} = \EE_{\boldsymbol{\theta}}\left\lbrace \frac{\partial}{\partial \theta_r} \ell(\boldsymbol{\theta};\boldsymbol{X})\frac{\partial}{\partial \theta_s}\ell(\boldsymbol{\theta};\boldsymbol{X})\right\rbrace = -\EE_{\boldsymbol{\theta}}\left\lbrace \frac{\partial^2}{\partial \theta_r \partial \theta_s} \ell(\boldsymbol{\theta};\boldsymbol{X}) \right\rbrace.
\end{equation*}
This condition implies that $nI(\boldsymbol{\theta})$ is the covariance matrix of $\nabla(\ell(\boldsymbol{\theta};\boldsymbol{x}))$.
We also assume that the submatrices in \eqref{Fisher} satisfy that $C$ is invertible and that $A-BC^{-1}B^T$ is positive definite. 
\end{itemize}

Under (R.C.1)-(R.C.5), \cite{Davison} proves that 
\begin{equation}
\label{Hoadley_normality}
\sqrt{n}\left(\boldsymbol{\hat{\theta}_n(X)} - \boldsymbol{\theta_0}\right) \xrightarrow[{n \to \infty}]{{\rm d}} \left[I(\boldsymbol{\theta_0})\right]^{-\frac{1}{2}}\boldsymbol{Z},
\end{equation}
where for $d$ fixed, $I_{d \times d}$ is the $d \times d$ identity matrix and $\boldsymbol{Z}\sim {\rm N}_d(\boldsymbol{0},I_{d \times d})$. In addition, $\xrightarrow[]{{\rm d}}$ denotes convergence in distribution. Under (R.C.1)-(R.C.5), \cite{anastasiou2015assessing} gives an explicit upper bound to the distance of the distribution of the MLE to the normal for i.i.d. random vectors; under stronger conditions bounds are also given for the case of non-identically distributed random elements. 

Some notation follows. Let the subscript $(m)\in \{ 1, 2, \ldots, d\} $ denote an index for which $\left|\hat{\theta}_n(\boldsymbol{x})_{(m)} - \theta_{0,(m)}\right|$ is the largest among the $d$ components;
\begin{align}
\nonumber & (m) \in \left\lbrace 1,\ldots,d\right\rbrace\;{\rm such\;that\;} \left|\hat{\theta}_n(\boldsymbol{x})_{(m)} - \theta_{0,(m)}\right| = \max_{j = 1,\ldots, d} \left|\hat{\theta}_n(\boldsymbol{x})_j - \theta_{0,j}\right|
\end{align}
and similarly $(m^*) \in \left\lbrace 1,2,\ldots, d-r \right\rbrace $ is defined with $\boldsymbol{\theta_0}$ replaced by $\boldsymbol{\theta^*} = \left(\theta_{*,1}, \ldots, \theta_{*,d-r}\right)$.
For our main result, we assume: 
\begin{itemize}
\item[(O1)] For any $\boldsymbol{\theta_0} \in \Theta$ there exists $\epsilon(\boldsymbol{\theta_0})>0$ and functions $M^*_{k^*j^*l^*}(\boldsymbol{x}), \forall k^*,j^*,l^* \in \left\lbrace 1,2,\ldots,d-r \right\rbrace$ such that for all $0 < \epsilon \le \epsilon(\boldsymbol{\theta_0})$ it holds that $\left|\frac{\partial^3}{\partial \theta_{k^*+r}\partial \theta_{j^*+r}\partial \theta_{l^*+r}}\ell(\boldsymbol{\theta^*},\boldsymbol{x})\right| \leq M^*_{k^*j^*l^*}(\boldsymbol{x})$ for all $\boldsymbol{\theta^*}$ such that $|\theta^*_j - \theta_{*,j}| < \epsilon$ $\forall j \in \left\lbrace 1,2,\ldots,d-r\right\rbrace$. 
\item [(O2)] With $M_{kjv}(\boldsymbol{x})$ as in (R.C.3) and $M^*_{k^*j^*l^*}(\boldsymbol{x})$ as in (O1), then for every $0 < \epsilon \le  \epsilon(\boldsymbol{\theta_0})$ and for $\kappa =2,4$,
\begin{align}
\nonumber & \mathbb{E}\left(\left(M_{kjv}(\boldsymbol{X})\right)^\kappa\middle|\left|\hat{\theta}_n(\boldsymbol{x})_{(m)} - \theta_{0,(m)} \right| < \epsilon \right) < \infty\\
\nonumber & \mathbb{E}\left(\left(M^*_{k^*j^*v^*}(\boldsymbol{X})\right)^\kappa\middle|\left|\hat{\theta}_{*}(\boldsymbol{x})_{(m^*)} - \theta_{*,(m^*)}\right| < \epsilon \right) < \infty.
\end{align}
\item  [(O3)] The random variables 
$Y_{i,j}(\boldsymbol{\theta}) = \frac{\partial}{\partial \theta_i} \log f(\boldsymbol{X_j}| \boldsymbol{\theta})$
have finite absolute moments up to fifth order. 
\end{itemize} 
Condition (O1) can be read as the version of (R.C.3), under the null hypothesis, where $\theta_{0,j} = 0$ for $j=1,2,\ldots,r$. Our main result is as follows.
\begin{theorem}
\label{Theorem_i.n.i.d}
Let $\boldsymbol{X_1}, \boldsymbol{X_2}, \ldots, \boldsymbol{X_n}$ be i.i.d. $\mathbb{R}^t$-valued, $t \in \mathbb{Z}^+$, random vectors with probability density (or mass) function $f(\boldsymbol{x_1}|\boldsymbol{\theta})$, for which the parameter space $\Theta$ is an open subset of $\mathbb{R}^d$. Assume that the MLE exists and is unique and that (R.C.1)-(R.C.5) as well as (O1), (O2) and (O3) are satisfied. Then for $- 2\log \Lambda$ as in \eqref{loglikelihoodsplitT1T2}, $h \in \C_b^2(\mathbb{R})$ and $K \sim \chi^2_r$, it holds that
\begin{align}
\label{final_bound_regression}
\nonumber \left|\mathbb{E}\left[h\left(-2\log \Lambda\right)\right]- \mathbb{E}[h(K)]\right| & \leq  2  \frac{(||h' || + || h''||)}{\sqrt{n}} R\left(  \left(\boldsymbol{\xi},\boldsymbol{\eta}\right), (A - B C^{-1} B^T)^{-1}, B C^{-1} \right)\\
& + \frac{1}{\sqrt{n}}\left(\|h'\|(K_1(\boldsymbol{\theta_0}) + K^*_1(\boldsymbol{\theta_0})) + K_2(\boldsymbol{\theta_0}) + K^*_2(\boldsymbol{\theta_0})\right),
\end{align}
where $ R\left(  \left(\boldsymbol{\xi},\boldsymbol{\eta}\right); \cdot, \cdot \right)$ will be defined in Equation  \eqref{RWLD} in Proposition \ref{gauntreinertcor}.
  In addition, for $0 < \epsilon \le  \epsilon( \boldsymbol{\theta_0})$, 
\begin{align}
\label{K1} & K_1(\boldsymbol{\theta_0}) = 3n\sum_{j=1}^d\sum_{k=1}^d\left[\mathbb{E}\left(Q_j^2Q_k^2\right)\right]^{\frac{1}{2}}\left[{\rm Var}\left(\frac{\partial^2}{\partial\theta_j\partial\theta_k}\log f(\boldsymbol{X_1}|\boldsymbol{\theta_0})\right)\right]^{\frac{1}{2}}\\
\nonumber
& +\sum_{l=1}^{d}\sum_{m=1}^{d}\left[\left[I(\boldsymbol{\theta_0})\right]^{-1}\right]_{lm}\sum_{j=1}^{d}\sum_{k=1}^{d}\sqrt{{\rm Var}\left(\frac{\partial^2}{\partial\theta_l\partial\theta_j}\log f(X_1|\boldsymbol{\theta_0})\right)}\left[\mathbb{E}\left(Q_j^6\right)\mathbb{E}\left(Q_k^6\right)\mathbb{E}\left(T_{mk}^6\right)\right]^{\frac{1}{6}}\end{align}
and
\begin{align}
\label{K2}& K_2(\boldsymbol{\theta_0}) = 2\sqrt{n}\frac{\|h\|}{\epsilon^2}\mathbb{E}\left(\sum_{j=1}^{d}Q_j^2\right)\\
\nonumber & + \sqrt{n}\|h'\|\frac{7}{3}\sum_{j=1}^d\sum_{k=1}^d\sum_{l=1}^d\left[\mathbb{E}\left(Q_j^2Q_k^2Q_l^2\right)\right]^{\frac{1}{2}}\left[\mathbb{E}\left[\left(M_{jkl}(\boldsymbol{X})\right)^2\middle|\left|Q_{(m)}\right|<\epsilon\right]\right]^{\frac{1}{2}}\\
\nonumber & + \frac{\|h'\|}{\sqrt{n}}\sum_{q=1}^{d}\sum_{k=1}^{d}\left|\left[\left[I(\boldsymbol{\theta_0})\right]^{-1}\right]_{kq}\right|\sum_{j=1}^{d}\sum_{l=1}^{d}\sum_{s=1}^{d}\sqrt{\mathbb{E}\left(Q_j^2Q_l^2Q_s^2\right)}\left[\mathbb{E}\left(T_{kj}^4\middle|\left|Q_{(m)}\right|<\epsilon\right)\right]^{\frac{1}{4}}\\
\nonumber &\qquad\qquad\times\left[\mathbb{E}\left(\left(M_{qsl}(\boldsymbol{X})\right)^4\middle|\left|Q_{(m)}\right|<\epsilon\right)\right]^{\frac{1}{4}}\\ 
\nonumber & + \frac{\|h'\|}{4\sqrt{n}}\sum_{b=1}^{d}\sum_{k=1}^{d}\sum_{s=1}^{d}\sum_{q=1}^{d}\sum_{l=1}^{d}\sum_{j=1}^{d}\left|\left[\left[I(\boldsymbol{\theta_0})\right]^{-1}\right]_{qb}\right|\sqrt{\mathbb{E}\left(Q_k^2Q_s^2Q_j^2Q_l^2\right)}\\ \nonumber 
& \qquad\quad \times \left[\mathbb{E}\left(\left(M_{bsk}(\boldsymbol{X})\right)^4\middle|\left|Q_{(m)}\right|<\epsilon\right)\right]^{\frac{1}{4}}\left[\mathbb{E}\left(\left(M_{qjl}(\boldsymbol{X})\right)^4\middle|\left|Q_{(m)}\right|<\epsilon\right)\right]^{\frac{1}{4}}.
\end{align}
and $K_1^*(\boldsymbol{\theta_0}), K_2^*(\boldsymbol{\theta_0})$ are the versions of $K_1(\boldsymbol{\theta_0})$ and $K_2(\boldsymbol{\theta_0})$, respectively,  under the null hypothesis. Here  
\begin{equation}
\begin{aligned}
\label{cm}
& Q_{j}  = Q_j(\boldsymbol{X},\boldsymbol{\theta_0}) := \hat{\theta}_n(\boldsymbol{X})_{j} - \theta_{0,j}, \forall j =1,2,\ldots,d\\
& Q^*_{j} = Q^{*}_j(\boldsymbol{X},\boldsymbol{\theta_*}) := \hat{\theta}^*(\boldsymbol{X})_{j} - \theta^*_{j}, \forall j =1,2,\ldots,d-r,\\
& T_{lj} = T_{lj}\left(\boldsymbol{\theta_0},\boldsymbol{X}\right) = \frac{\partial^2}{\partial\theta_l\partial\theta_j}\ell(\boldsymbol{\theta_0};\boldsymbol{X}) + n[I(\boldsymbol{\theta_0})]_{lj}, \quad j,l \in \left\lbrace 1,2,\ldots,d \right\rbrace,\\
& T^*_{lj} = T^*_{lj}\left(\boldsymbol{\theta_0},\boldsymbol{X}\right) = \frac{\partial^2}{\partial\theta_{l+r}\partial\theta_{j+r}}\ell(\boldsymbol{\theta_0};\boldsymbol{X}) + nC_{lj}, \quad j,l \in \left\lbrace 1,2,\ldots,d-r \right\rbrace.
\end{aligned}
\end{equation}
\end{theorem}
\vspace{0.05in}
\begin{remark}
\label{remark_multi_order}
\textbf{(1)} At first glance, the bound seems complicated. However, the examples that follow show that the terms are easily calculated.\\
\textbf{(2)} For $Q_j$ as in \eqref{cm}, we have that $\mathbb{E}\left(Q_j^2\right) = \mathcal{O}\left(\frac{1}{n}\right)$. To see this better, use that from the asymptotic normality of the MLE in \eqref{Hoadley_normality}, $\sqrt{n}\mathbb{E}\left(\boldsymbol{\hat{\theta}_n(X)} - \boldsymbol{\theta_0}\right) \xrightarrow[{n \to \infty}]{{}} \boldsymbol{0}$ and therefore
\begin{equation}
\nonumber \mathbb{E}\left(Q_j\right) = o\left(\frac{1}{\sqrt{n}}\right),\; \forall j \in \left\lbrace 1,2, \ldots, d\right\rbrace.
\end{equation}
In addition, ${\rm Cov}\left(\sqrt{n}\left[I(\boldsymbol{\theta_0})\right]^{\frac{1}{2}}\left(\boldsymbol{\hat{\theta}_n(X)} - \boldsymbol{\theta_0}\right)\right) \xrightarrow[{n \to \infty}]{{}} I_{d\times d}$. Therefore, since $I(\boldsymbol{\theta_0})$ is a symmetric matrix, we have that
\begin{equation}
\label{ordervariancei.n.i.d}
n\left[I(\boldsymbol{\theta_0})\right]^{\frac{1}{2}}{\rm Cov}\left(\boldsymbol{\hat{\theta}_n(X)}\right)\left[I(\boldsymbol{\theta_0})\right]^{\frac{1}{2}} \xrightarrow[{n \to \infty}]{{}} I_{d \times d}.
\end{equation}
It follows from \eqref{ordervariancei.n.i.d} that $${\rm Var}\left(\hat{\theta}_n(\boldsymbol{X})_j\right) = \mathcal{O}\left(\frac{1}{n}\right),\;\forall j \in \left\lbrace 1,2,\ldots,d \right\rbrace.$$
Combining these results,
\begin{equation}
\label{orderMSEi.n.i.d}
\mathbb{E}\left(Q_j^2\right) = {\rm Var}\left(\hat{\theta}_n(\boldsymbol{X})_j\right) + \left[{\rm E}\left(Q_j\right)\right]^2 = \mathcal{O}\left(\frac{1}{n}\right).
\end{equation}
\textbf{(3)} With $T_{lj}$ as in \eqref{cm}, using (R.C.5) and the fact that $\boldsymbol{X_1}, \boldsymbol{X_2}, \ldots, \boldsymbol{X_n}$ are i.i.d. yields
\begin{align}
\label{orderTlj}
\nonumber \mathbb{E}\left(T_{lj}^2\right) & = \mathbb{E}\left(\frac{\partial^2}{\partial\theta_l\partial\theta_j}\ell(\boldsymbol{\theta_0};\boldsymbol{X}) + n[I(\boldsymbol{\theta_0})]_{lj}\right)^2 = {\rm Var}\left(\frac{\partial^2}{\partial\theta_l\partial\theta_j}\ell(\boldsymbol{\theta_0};\boldsymbol{X})\right)\\
& = n{\rm Var}\left(\frac{\partial^2}{\partial\theta_l\partial\theta_j}\log\left(f(\boldsymbol{X_1}|\boldsymbol{\theta_0})\right)\right),
\end{align}
showing that $\mathbb{E}\left(T_{lj}^2\right)$ is $\mathcal{O}(n)$.\\
\textbf{(4)} For fixed $d$, the upper bound we give in \eqref{final_bound_regression} is $\mathcal{O}\left(n^{-1/2}\right)$. The expression for  \linebreak $R\left((\boldsymbol{\xi}, \boldsymbol{\eta})), (A - B C^{-1} B^T)^{-1}, B C^{-1}\right)$ given in Proposition \ref{gauntreinertcor}, is $\mathcal{O}(1)$. In addition, using \eqref{orderMSEi.n.i.d} and \eqref{orderTlj} it can be deduced that
\begin{equation}
\nonumber K_1(\boldsymbol{\theta_0}) = \mathcal{O}(1),\;\;
K_1^*(\boldsymbol{\theta_0}) = \mathcal{O}(1),\;\;
K_2(\boldsymbol{\theta_0}) = \mathcal{O}(1),\;\; K_2^*(\boldsymbol{\theta_0}) = \mathcal{O}(1).
\end{equation}
Hence, the upper bound in Theorem \ref{Theorem_i.n.i.d} is $\mathcal{O}\left(n^{-1/2}\right)$.\\
\textbf{(5)} If the dimensionality of the parameter is not fixed but if the entries of $I(\boldsymbol{\theta_0})$ are of order 1, then  $c$ in \eqref{RWLD} is of order $r^2(d-r)$.  
Therefore, we obtain that the first term of the bound is $\mathcal{O}\left(r^2 (d-r) d^4 n^{-1/2}\right)$. Using \eqref{orderMSEi.n.i.d} and \eqref{orderTlj}, we have that the second and third terms (related to $K_1(\boldsymbol{\theta_0})$ and $K_1^*(\boldsymbol{\theta_0})$) of the bound are of order $d^2n^{-1/2}$, while the fourth and fifth terms (related to $K_2(\boldsymbol{\theta_0})$ and $K_2^*(\boldsymbol{\theta_0})$) are both $\mathcal{O}\left(d^3 n^{-1/2}\right)$. Hence, the overall order of the bound in the chisquare approximation for the likelihood ratio test is $d^2 \max( d,  r^2 d^2(d-r) )n^{-1/2}$ which is at most of order $d^7 n^{-1/2}$. The chisquare approximation is justified when $ d = o(n^{1/14})$.
\\
\textbf{(6)} Due to the smoothness assumptions in this paper, the bound in  Theorem \ref{Theorem_i.n.i.d} is not given in a standard probability distance. Instead it could be re-phrased in terms of the integral probability metric 
$$ d (\mu, \nu) = \sup_{h \in \C_b^2(\mathbb{R}): || h'|| \le 1, || h''|| \le 1} | \EE h(X) - \EE h(Y)|$$
where $X \sim \mu$ and $Y \sim \nu$. For more details on such metrics see for example \cite{zolotarev1983probability} and \cite{gibbs2002choosing}. 
 \end{remark}

\subsection{Proof of Theorem \ref{Theorem_i.n.i.d} } 

The proof of Theorem \ref{Theorem_i.n.i.d}  follows the outline from the introduction by bounding the error in each approximation.
The log-likelihood ratio statistic can be expressed as in \eqref{loglikelihoodsplitT1T2}. The expected Fisher information matrix is given in \eqref{Fisher}; 
with $C^{-1}$ assumed to exist. From now on, we will be using the notation introduced in \eqref{cm}. 

The different steps are disentangled into results which hold for every realisation $\boldsymbol{x}$ and results which hold when taking expectations over test functions. 

\subsubsection{Quantifying  the approximation for $2 \log T_1$ and for $2 \log T_2$}

The first step in the proof is to quantify the approximation \eqref{t1exp0} for $2 \log T_1.$ 

\begin{proposition}\label{t1} 
Under the assumptions of Theorem \ref{Theorem_i.n.i.d}, 
\begin{equation}
\nonumber 2\log T_1 = n(\boldsymbol{\hat{\theta}_n(X)} - \boldsymbol{\theta_0})^{\intercal}I(\boldsymbol{\theta_0})(\boldsymbol{\hat{\theta}_n(X)} - \boldsymbol{\theta_0}) + R_1 
+R_2 ,
\end{equation}
where using the notation in \eqref{cm},
\begin{equation}
\nonumber R_1 = R_1(\boldsymbol{X},\boldsymbol{\theta_0}) = - \sum_{j=1}^d\sum_{k=1}^dQ_jQ_kT_{kj} 
\end{equation}
and
\begin{equation}
\nonumber R_2 =  R_2(\boldsymbol{X},\boldsymbol{\theta_0}) = \sum_{j=1}^d\sum_{k=1}^d\sum_{s=1}^dQ_jQ_kQ_s\frac{\partial^3}{\partial\theta_j\partial\theta_k\partial\theta_s}\left(\frac{1}{3}\ell\left(\boldsymbol{\tilde{\theta}};\boldsymbol{X}\right) - \ell\left(\boldsymbol{\tilde{\tilde{\theta}}};\boldsymbol{X}\right)\right)
\end{equation} 
for some $\boldsymbol{\tilde{\theta}}, \boldsymbol{\tilde{\tilde{\theta}}}$ between $\boldsymbol{\hat{\theta}_n(X)}$ and $\boldsymbol{\theta_0}$.
\end{proposition}


\begin{proof} 
The regularity conditions and a third order Taylor expansion of $\ell(\boldsymbol{\theta_0};\boldsymbol{x})$ about $\boldsymbol{\hat{\theta}_n(x)}$ yield
\begin{align}
\nonumber & \ell(\boldsymbol{\theta_0};\boldsymbol{x}) = \ell\left(\boldsymbol{\hat{\theta}_n(x)};\boldsymbol{x}\right) - \sum_{j=1}^dQ_j\frac{\partial}{\partial\theta_j}\ell\left(\boldsymbol{\hat{\theta}_n(x)};\boldsymbol{x}\right)\\
\nonumber & + \frac{1}{2}\sum_{j=1}^d\sum_{k=1}^dQ_jQ_k\frac{\partial^2}{\partial\theta_k\partial\theta_j}\ell\left(\boldsymbol{\hat{\theta}_n(x)};\boldsymbol{x}\right) - \frac{1}{6}\sum_{j=1}^d\sum_{k=1}^d\sum_{s=1}^dQ_jQ_kQ_s\frac{\partial^3}{\partial\theta_j\partial\theta_k\partial\theta_s}\ell\left(\boldsymbol{\tilde{\theta}};\boldsymbol{x}\right),
\end{align}
where $\boldsymbol{\tilde{\theta}}$ is between $\boldsymbol{\theta_0}$ and $\boldsymbol{\hat{\theta}_n(x)}$. As $\frac{\partial}{\partial \theta_j}\ell(\boldsymbol{\hat{\theta}_n(x)};\boldsymbol{x}) = 0$, $\forall j \in \left\lbrace  1,2,\ldots,d \right\rbrace$,
\begin{align}
\nonumber & 2\log T_1 = - \sum_{j=1}^d\sum_{k=1}^dQ_jQ_k\frac{\partial^2}{\partial\theta_k\partial\theta_j}\ell\left(\boldsymbol{\hat{\theta}_n(X)};\boldsymbol{X}\right)\\
\nonumber & \qquad\quad\;\;\; + \frac{1}{3}\sum_{j=1}^d\sum_{k=1}^d\sum_{s=1}^dQ_jQ_kQ_s\frac{\partial^3}{\partial\theta_j\partial\theta_k\partial\theta_s}\ell\left(\boldsymbol{\tilde{\theta}};\boldsymbol{X}\right)\\
\nonumber & = n\sum_{j=1}^d\sum_{k=1}^d\left[I(\boldsymbol{\theta_0})\right]_{kj}Q_kQ_j - \sum_{j=1}^d\sum_{k=1}^dQ_jQ_k\left(\frac{\partial^2}{\partial\theta_k\partial\theta_j}\ell\left(\boldsymbol{\hat{\theta}_n(X)};\boldsymbol{X}\right) + n\left[I(\boldsymbol{\theta_0})\right]_{kj}\right)\\
\nonumber & \qquad + \frac{1}{3}\sum_{j=1}^d\sum_{k=1}^d\sum_{s=1}^dQ_jQ_kQ_s\frac{\partial^3}{\partial\theta_j\partial\theta_k\partial\theta_s}\ell\left(\boldsymbol{\tilde{\theta}};\boldsymbol{X}\right).
\end{align}
Then, a first order Taylor expansion of $\frac{\partial^2}{\partial\theta_j\partial\theta_k}\ell\left(\boldsymbol{\hat{\theta}_n(x)};\boldsymbol{x}\right)$ about $\boldsymbol{\theta_0}$ yields
\begin{align}
\label{Taylor_useful2}
\nonumber 2\log T_1 &= n\sum_{j=1}^d\sum_{k=1}^d\left[I(\boldsymbol{\theta_0})\right]_{kj}Q_kQ_j + R_1 + R_2\\
& = n(\boldsymbol{\hat{\theta}_n(X)} - \boldsymbol{\theta_0})^{\intercal}I(\boldsymbol{\theta_0})(\boldsymbol{\hat{\theta}_n(X)} - \boldsymbol{\theta_0}) + R_1 + R_2.
\end{align}
This finishes the proof. \end{proof} 

From Proposition \ref{t1} the next approximation of the log-likelihood ratio is almost immediate. Using  \eqref{loglikelihoodsplitT1T2}, 
 Proposition \ref{t1} and its analogous expression for $2 \log T_2$  with $\boldsymbol{\hat{\theta}_n(x)}$ replaced by $\boldsymbol{\hat{\theta}^{{\rm res}}(x)}$, 
\begin{align}
\label{cor:llr}  - 2 \log \Lambda & =
 n(\boldsymbol{\hat{\theta}_n(X)} - \boldsymbol{\theta_0})^{\intercal}I(\boldsymbol{\theta_0})(\boldsymbol{\hat{\theta}_n(X)} - \boldsymbol{\theta_0}) + R_1 + R_2\\
\nonumber & \quad-  n(\boldsymbol{\hat{\theta}^{{\rm res}}(X)} - \boldsymbol{\theta_0})^{\intercal}C (\boldsymbol{\hat{\theta}^{{\rm res}}(X)} - \boldsymbol{\theta_0}) - R_1^* - R_2^*,
\end{align}
where $R_1$ and $R_2$ are as in Proposition \ref{t1} and 
 $R_1^*$ and $R_2^*$ are the corresponding expressions with $\boldsymbol{\hat{\theta}_n(x)}$ replaced by $\boldsymbol{\hat{\theta}^{{\rm res}}(x)}$ .

\subsubsection{Quantification of the approximation for the score function} 

\begin{proposition}\label{scorefct} 
Under the assumptions of Theorem \ref{Theorem_i.n.i.d}, 
\begin{align}
\label{ratio_expression_2}
& n(\boldsymbol{\hat{\theta}_n(X)} - \boldsymbol{\theta_0})^{\intercal}I(\boldsymbol{\theta_0})(\boldsymbol{\hat{\theta}_n(X)} - \boldsymbol{\theta_0}) = \begin{pmatrix}\boldsymbol{\xi} \\ 
\boldsymbol{\eta}
\end{pmatrix}^{\intercal}[I(\boldsymbol{\theta_0})]^{-1}\begin{pmatrix}\boldsymbol{\xi} \\
\boldsymbol{\eta}
\end{pmatrix}\\ \nonumber 
& \qquad - \left(\boldsymbol{R_3} + \boldsymbol{R_4}\right)^{\intercal}\left(\boldsymbol{R_3} + \boldsymbol{R_4}\right) + 2\sqrt{n}\left(\boldsymbol{\hat{\theta}_n(X)} - \boldsymbol{\theta_0}\right)^{\intercal}[I(\boldsymbol{\theta_0})]^{\frac{1}{2}}(\boldsymbol{R_3} + \boldsymbol{R_4}),
\end{align}
with 
\begin{equation}
\label{R3R4}
\begin{aligned}
& \boldsymbol{R_3} = \boldsymbol{R_3(X,\theta_0)} = \frac{1}{\sqrt{n}}[I(\boldsymbol{\theta_0})]^{-\frac{1}{2}}\sum_{j=1}^{d}Q_j\left(\nabla \left(\frac{\partial}{\partial \theta_j}\ell(\boldsymbol{\theta_0};\boldsymbol{X})\right) + n[I(\boldsymbol{\theta_0})]_{[j]}\right)\\
& \boldsymbol{R_4} = \boldsymbol{R_4(X,\theta_0)} = \frac{1}{2\sqrt{n}}[I(\boldsymbol{\theta_0})]^{-\frac{1}{2}}\sum_{j=1}^{d}\sum_{q=1}^{d}Q_jQ_q\left(\nabla\left(\frac{\partial^2}{\partial\theta_j\partial\theta_q}\ell(\boldsymbol{\theta};\boldsymbol{X})\Big|_{\substack{\boldsymbol{\theta} = \boldsymbol{\theta_0^{*}}}}\right)\right).
\end{aligned}
\end{equation}
for some $\boldsymbol{\theta_0^*}$ between $\boldsymbol{\theta_0}$ and $\boldsymbol{\hat{\theta}_n(x)}$.
\end{proposition}

\begin{proof}
For $\boldsymbol{\theta_0^*}$ between $\boldsymbol{\theta_0}$ and $\boldsymbol{\hat{\theta}_n(x)}$, using results from \cite{anastasiou2015assessing}, we have that
\begin{align}
\nonumber n\sum_{j=1}^{d}[I(\boldsymbol{\theta_0})]_{kj}Q_j &= \frac{\partial}{\partial\theta_k}\ell(\boldsymbol{\theta_0};\boldsymbol{x})+ \sum_{j=1}^{d}Q_jT_{kj}\\
\nonumber & \quad + \frac{1}{2}\sum_{j=1}^{d}\sum_{q=1}^{d}Q_jQ_q\left(\frac{\partial^3}{\partial\theta_k\partial\theta_j\partial\theta_q}\ell(\boldsymbol{\theta};\boldsymbol{x})\Big|_{\substack{\boldsymbol{\theta} = \boldsymbol{\theta_0^{*}}}}\right).
\end{align}
which holds $\forall k \in \left\lbrace 1,2,\ldots,d \right\rbrace$ and therefore
\begin{align}
\nonumber & \sqrt{n}[I(\boldsymbol{\theta_0})]^{\frac{1}{2}}(\boldsymbol{\hat{\theta}_n(x)} - \boldsymbol{\theta_0})\\
\nonumber & = \frac{1}{\sqrt{n}}[I(\boldsymbol{\theta_0})]^{-\frac{1}{2}}\left\lbrace\vphantom{(\left(\sup_{\theta:|\theta-\theta_0|\leq\epsilon}\left|l^{(3)}(\theta;\boldsymbol{X})\right|\right)^2}\nabla(\ell(\boldsymbol{\theta_0};\boldsymbol{x})) +  \sum_{j=1}^{d}Q_j\left(\nabla \left(\frac{\partial}{\partial \theta_j}\ell(\boldsymbol{\theta_0};\boldsymbol{x})\right) + n[I(\boldsymbol{\theta_0})]_{[j]}\right)\vphantom{(\left(\sup_{\theta:|\theta-\theta_0|\leq\epsilon}\left|l^{(3)}(\theta;\boldsymbol{X})\right|\right)^2}\right\rbrace\\
\nonumber & \;\;+ \frac{1}{2\sqrt{n}}[I(\boldsymbol{\theta_0})]^{-\frac{1}{2}}\left\lbrace\vphantom{(\left(\sup_{\theta:|\theta-\theta_0|\leq\epsilon}\left|l^{(3)}(\theta;\boldsymbol{X})\right|\right)^2}\sum_{j=1}^{d}\sum_{q=1}^{d}Q_jQ_q\left(\nabla\left(\frac{\partial^2}{\partial\theta_j\partial\theta_q}\ell(\boldsymbol{\theta};\boldsymbol{x})\Big|_{\substack{\boldsymbol{\theta} = \boldsymbol{\theta_0^{*}}}}\right)\right)\vphantom{(\left(\sup_{\theta:|\theta-\theta_0|\leq\epsilon}\left|l^{(3)}(\theta;\boldsymbol{X})\right|\right)^2}\right\rbrace,
\end{align}
where $[I(\boldsymbol{\theta_0})]_{[j]}$ is the $j^{th}$ column of the matrix $I(\boldsymbol{\theta_0})$. Using the score vector notation \eqref{score}, we have that
\begin{equation}
\label{mid_result_with_R3_R4}
\sqrt{n}[I(\boldsymbol{\theta_0})]^{\frac{1}{2}}(\boldsymbol{\hat{\theta}_n(X)} - \boldsymbol{\theta_0}) = [I(\boldsymbol{\theta_0})]^{-\frac{1}{2}}\begin{pmatrix}\boldsymbol{\xi} \\
\boldsymbol{\eta}
\end{pmatrix} + \boldsymbol{R_3} + \boldsymbol{R_4},
\end{equation}
where $\boldsymbol{R_3}$ and $\boldsymbol{R_4}$ are as in \eqref{R3R4}. Using \eqref{mid_result_with_R3_R4} and  that $I(\boldsymbol{\theta_0})$ is a symmetric matrix leads  to
\begin{align}
\label{ratio_expression_1}
& n(\boldsymbol{\hat{\theta}_n(X)} - \boldsymbol{\theta_0})^{\intercal}I(\boldsymbol{\theta_0})(\boldsymbol{\hat{\theta}_n(X)} - \boldsymbol{\theta_0})\\
&\nonumber  = \begin{pmatrix}\boldsymbol{\xi} \\
\boldsymbol{\eta}
\end{pmatrix}^{\intercal}[I(\boldsymbol{\theta_0})]^{-1}\begin{pmatrix}\boldsymbol{\xi} \\
\boldsymbol{\eta}
\end{pmatrix} + \left(\boldsymbol{R_3} + \boldsymbol{R_4}\right)^{\intercal}\left(\boldsymbol{R_3} + \boldsymbol{R_4}\right) + 2\begin{pmatrix}\boldsymbol{\xi} \\ \nonumber
\boldsymbol{\eta}
\end{pmatrix}^{\intercal}[I(\boldsymbol{\theta_0})]^{-\frac{1}{2}}\left(\boldsymbol{R_3} + \boldsymbol{R_4}\right).
\end{align}
However, from \eqref{mid_result_with_R3_R4},
\begin{equation}
\nonumber \begin{pmatrix}\boldsymbol{\xi} \\
\boldsymbol{\eta}
\end{pmatrix}^{\intercal}[I(\boldsymbol{\theta_0})]^{-\frac{1}{2}}= \sqrt{n}\left(\boldsymbol{\hat{\theta}_n(X)} - \boldsymbol{\theta_0}\right)^{\intercal}[I(\boldsymbol{\theta_0})]^{\frac{1}{2}} - (\boldsymbol{R_3} + \boldsymbol{R_4})^{\intercal},
\end{equation}
so that
\begin{align}
\nonumber \begin{pmatrix}\boldsymbol{\xi} \\
\boldsymbol{\eta}
\end{pmatrix}^{\intercal}[I(\boldsymbol{\theta_0})]^{-\frac{1}{2}}\left(\boldsymbol{R_3} + \boldsymbol{R_4}\right) & = \sqrt{n}\left(\boldsymbol{\hat{\theta}_n(X)} - \boldsymbol{\theta_0}\right)^{\intercal}[I(\boldsymbol{\theta_0})]^{\frac{1}{2}}\left(\boldsymbol{R_3} + \boldsymbol{R_4}\right)\\
\nonumber & \quad - (\boldsymbol{R_3} + \boldsymbol{R_4})^{\intercal}\left(\boldsymbol{R_3} + \boldsymbol{R_4}\right).
\end{align}
Using this in \eqref{ratio_expression_1} yields the assertion. 
\end{proof}

A similar result holds for $T_2$. Following exactly the same steps as for  Proposition \eqref{scorefct}, but now with $\boldsymbol{\theta^*}$ instead of $\boldsymbol{\theta_0}$,
\begin{align}
\label{final_expression_T2}
\nonumber 2\log T_2 &=  \boldsymbol{\eta}^{\intercal}C^{-1}\boldsymbol{\eta}
 -{R}_1^* +{R}_2^*- \left(\boldsymbol{{R}_3^*} + \boldsymbol{{R}_4^*}\right)^{\intercal}\left(\boldsymbol{{R}_3^*} + \boldsymbol{{R}_4^*}\right)\\
& \qquad + 2\sqrt{n}\left(\boldsymbol{\hat{\theta}_*(X)} - \boldsymbol{\theta_*}\right)^{\intercal}C^{\frac{1}{2}}(\boldsymbol{{R}_3^*} + \boldsymbol{{R}_4^*}).
\end{align}
Combining \eqref{loglikelihoodsplitT1T2},  Proposition \ref{t1} and \eqref{final_expression_T2} gives 
\begin{align}\label{cor:key}
 & -2\log \Lambda = \begin{pmatrix}\boldsymbol{\xi} \\
\boldsymbol{\eta}
\end{pmatrix}^{\intercal}[I(\boldsymbol{\theta_0})]^{-1}\begin{pmatrix}\boldsymbol{\xi} \\
\boldsymbol{\eta}
\end{pmatrix} - \boldsymbol{\eta}^{\intercal}C^{-1}\boldsymbol{\eta} + R_{A_1} + R_{A_2}  + R_{B_1} + R_{B_2},
\end{align}
where
\begin{equation}
\begin{aligned}
\label{RARBterms}
& R_{A_1} = R_1 - \boldsymbol{R_3}^{\intercal}\boldsymbol{R_3} + 2\sqrt{n}\left(\boldsymbol{\hat{\theta}_n(X)} - \boldsymbol{\theta_0}\right)^{\intercal}[I(\boldsymbol{\theta_0})]^{\frac{1}{2}}\boldsymbol{R_3}\\
& R_{A_2} = -{R_1}^* + (\boldsymbol{{R}_3^*})^{\intercal}\boldsymbol{{R}_3^*} - 2\sqrt{n}\left(\boldsymbol{\hat{\theta}_*(X)} - \boldsymbol{\theta_*}\right)^{\intercal}C^{\frac{1}{2}}\boldsymbol{{R}_3^*}\\
& R_{B_1} = R_2 - \boldsymbol{R_4}^{\intercal}(\boldsymbol{R_3} + \boldsymbol{R_4}) - \boldsymbol{R_3}^{\intercal}\boldsymbol{R_4} + 2\sqrt{n}\left(\boldsymbol{\hat{\theta}_n(X)} - \boldsymbol{\theta_0}\right)^{\intercal}\left[I(\boldsymbol{\theta_0})\right]^{\frac{1}{2}}\boldsymbol{R_4}\\
& R_{B_2} = -{R}_2^* + (\boldsymbol{{R}_4^*})^{\intercal}(\boldsymbol{{R}_3^*} + \boldsymbol{{R}_4}^*) + (\boldsymbol{{R}_3^*})^{\intercal}\boldsymbol{{R}_4^*} - 2\sqrt{n}\left(\boldsymbol{\hat{\theta}_*(X)} - \boldsymbol{\theta_*}\right)^{\intercal}C^{\frac{1}{2}}\boldsymbol{{R}_4^*}.
\end{aligned}
\end{equation}
Here $R_1, R_2$ are as in Proposition \ref{t1} and $R_1^*, R_2^*$ are their respective versions when we work under the null hypothesis. Furthermore, $\boldsymbol{{R}_3}$ and $\boldsymbol{{R}_4}$ are as in Proposition \ref{scorefct}, and  $\boldsymbol{{R_3}^*}$ and $\boldsymbol{{R_4}^*}$ are the
corresponding remainder terms  from Proposition \ref{scorefct} with $\boldsymbol{\hat{\theta}_n(x)}$ replaced by $\boldsymbol{\hat{\theta}^{{\rm res}}(x)}$.
Note that $R_{A_1}$ and $R_{B_1}$ contain the terms that are obtained through $2\log T_1$, whereas $R_{A_2}$ and $R_{B_2}$ contain the quantities that are due to $2\log T_2$.

From now on, let 
\begin{equation}
\label{g_xi_eta}
g(\boldsymbol{\xi},\boldsymbol{\eta}) = \begin{pmatrix}\boldsymbol{\xi} \\
\boldsymbol{\eta}
\end{pmatrix}^{\intercal}[I(\boldsymbol{\theta_0})]^{-1}\begin{pmatrix}\boldsymbol{\xi} \\
\boldsymbol{\eta}
\end{pmatrix} - \boldsymbol{\eta}^{\intercal}C^{-1}\boldsymbol{\eta}.
\end{equation}
It is straightforward to simplify this expression to give 
\begin{equation} \label{lem:simple} 
g(\boldsymbol{\xi},\boldsymbol{\eta}) = ( \boldsymbol{\xi}- B C^{-1} \boldsymbol{\eta})^{\intercal} \left( A - B C^{-1} B^{\intercal}\right)^{-1} ( \boldsymbol{\xi}- B C^{-1} \boldsymbol{\eta}).
\end{equation}

\subsubsection{Quantifying the Chisquare approximation} 

Now we use a multivariate $r$-dimensional normal approximation for $ \boldsymbol{\xi}- B C^{-1} \boldsymbol{\eta}$ which is based on the asymptotic normality of $(\boldsymbol{\xi}, \boldsymbol{\eta})^{\intercal}$ (of dimension $d$). 
Noting that for any positive semidefinite $r \times r$ matrix $U$,
$( \boldsymbol{\xi}- B C^{-1} \boldsymbol{\eta})^{\intercal} U ( \boldsymbol{\xi}- B C^{-1} \boldsymbol{\eta})$ is a quadratic form and therefore we can  apply results from \cite{Gaunt_Reinert} for symmetric test functions to obtain an overall bound to the chisquare distribution with $r$ degrees of freedom. 
Here is the more detailed setup. Let $\boldsymbol{Z}$ have the standard $d$-dimensional multivariate normal distribution, so that for a positive definite $d \times d$ matrix $\Sigma = (\sigma_{j,k}), j,k=1,2,\ldots,d$, we have  $\Sigma^{1/2}\boldsymbol{Z}\sim\mathrm{MVN}(\mathbf{0},\Sigma)$. For a matrix $A \in \mathbb{R}^{k_1 \times k_2}$, we use  the $L_\infty$-matrix norm
$$ ||| A |||_\infty := \max_{1 \le i \le k_1} \sum_{j=1}^{k_2} | A_{i,j}|.$$

\begin{proposition} \label{gauntreinertcor}  Let $Y_{i,j}(\boldsymbol{\theta}) = \frac{\partial}{\partial \theta_i} \log f(\boldsymbol{X_j}| \boldsymbol{\theta})$,  and $\boldsymbol{W} = (\boldsymbol{\xi}, \boldsymbol{\eta})$ with 
$$\xi_i = \frac{1}{\sqrt{n}} \sum_{j=1}^n Y_{i,j}(\boldsymbol{\theta_0}), \;\; i=1, \ldots, r;  \quad 
\eta_i = \frac{1}{\sqrt{n}} \sum_{j=1}^n Y_{i,j}(\boldsymbol{\theta_0}), \;\; i=r+1, \ldots, d. 
$$
For $D$ an $r \times (d-r)$ matrix and $U$ a  positive semidefinite $r \times r$ matrix, set
$$g(\boldsymbol{x}, \boldsymbol{y})  = ( \boldsymbol{x}- D \boldsymbol{y})^{\intercal} U ( \boldsymbol{x}-D \boldsymbol{y}).$$
Then, under the assumptions of Theorem \ref{Theorem_i.n.i.d}, for $h\in C_b^2(\mathbb{R})$, 
\begin{equation}
\nonumber \left|\mathbb{E}\left[h(g(\boldsymbol{W}))\right]-\mathbb{E}\left[h\left(g\left(\left[I(\boldsymbol{\theta_0})\right]^{1/2}\boldsymbol{Z}\right)\right)\right]\right| \leq \frac{2 (||h' || + || h''||) }{\sqrt{n}} R(\boldsymbol{W}, U, D) 
\end{equation}
with 
\begin{align}
\label{RWLD}
\nonumber & R(\boldsymbol{W}, U, D) := \frac{c}{n} \min_{1\leq s\leq d}\mathbb{E}\left|\left(\left[I(\boldsymbol{\theta_0})\right]^{-1/2}\mathbf{Z}\right)_s\right|  \sum_{i=1}^n\sum_{j,k,l=1}^d\left\lbrace\vphantom{(\left(\sup_{\theta:|\theta-\theta_0|\leq\epsilon}\left|l^{(3)}(\theta;\boldsymbol{X})\right|\right)^2}\mathbb{E}\left|Y_{ji}Y_{ki}Y_{li}\right|\right.\\ 
\nonumber & \; \left. + {{8cd}}   \sum_{t=1}^d \bigg(4 \mathbb{E}\left|Y_{ji}Y_{ki}Y_{li}\right|\mathbb{E}\left(W_t^2\right) +\frac{4}{n}\mathbb{E}\left|Y_{ji}Y_{ki}Y_{li}Y_{ti}^{2}\right|+\frac{\mathbb{E}\left|\left(\left[I(\boldsymbol{\theta_0})\right]^{-1/2}\mathbf{Z}\right)_sZ_t^2\right|}{\mathbb{E}\left|\left(\left[I(\boldsymbol{\theta_0})\right]^{-1/2}\mathbf{Z}\right)_s\right|}\mathbb{E}\left|Y_{ji}Y_{ki}Y_{li}\right|\bigg)\right.\\
\nonumber &\; +\left. 2\left|\mathbb{E}\left(Y_{ji}Y_{ki}\right)\right|\left[\vphantom{(\left(\sup_{\theta:|\theta-\theta_0|\leq\epsilon}\left|l^{(3)}(\theta;\boldsymbol{X})\right|\right)^2}\mathbb{E}\left|Y_{li}\right|
+ {{16 cd}} \sum_{t=1}^d \left(\vphantom{(\left(\sup_{\theta:|\theta-\theta_0|\leq\epsilon}\left|l^{(3)}(\theta;\boldsymbol{X})\right|\right)^2}4\mathbb{E}\left|Y_{li}\right|\mathbb{E}\left(W_t^{2}\right)+\frac{4}{n}\mathbb{E}\left|Y_{li}Y_{ti}^{2}\right|\right.\right.\right.\\
&\qquad\qquad\qquad\qquad \left.\left.\left. +\frac{\mathbb{E}\left|\left(\left[I(\boldsymbol{\theta_0})\right]^{-1/2}\mathbf{Z}\right)_sZ_t^{2}\right|}{\mathbb{E}\left|\left(\left[I(\boldsymbol{\theta_0})\right]^{-1/2}\mathbf{Z}\right)_s\right|}\mathbb{E}\left|Y_{li}\right|\vphantom{(\left(\sup_{\theta:|\theta-\theta_0|\leq\epsilon}\left|l^{(3)}(\theta;\boldsymbol{X})\right|\right)^2}\right)\vphantom{(\left(\sup_{\theta:|\theta-\theta_0|\leq\epsilon}\left|l^{(3)}(\theta;\boldsymbol{X})\right|\right)^2}\right]\vphantom{(\left(\sup_{\theta:|\theta-\theta_0|\leq\epsilon}\left|l^{(3)}(\theta;\boldsymbol{X})\right|\right)^2}\right\rbrace.
\end{align}
Here,
\begin{equation}
\label{c}
c = c (U, D) =  \max \left\{  ||| U |||_\infty, ||| D^T U D |||_\infty, ||| D^T U |||_\infty \right\} .
\end{equation} 
\end{proposition}

\begin{remark}
\begin{enumerate}
\item 
Note that while \cite{Gauntetal} also provide a bound of order $n^{-1}$ for this approximation, the terms in the bound require higher moment conditions.  As the overall order of the bound in Theorem \ref{Theorem_i.n.i.d} is $n^{-1/2}$, we use the simpler bound from \cite{Gauntetal} here. 
\item Proposition \ref{gauntreinertcor} can also be applied to the score-test like statistic 
\begin{eqnarray}\label{scoretestexpression}
\begin{pmatrix}\boldsymbol{\xi} \\
\boldsymbol{\eta}
\end{pmatrix}^{\intercal}[I(\boldsymbol{\theta_0})]^{-1}\begin{pmatrix}\boldsymbol{\xi} \\
\boldsymbol{\eta}
\end{pmatrix} 
\end{eqnarray} 
which is closely related to the classical score test statistic in which $I(\boldsymbol{\theta_0})$ is replaced by $I(\boldsymbol{\hat{\theta}})$. 
For the statistic \eqref{scoretestexpression} the  order $n^{-1}$ bound in \cite{Gauntetal} would apply. Using Taylor expansion to assess $[I(\boldsymbol{\theta_0})]^{-1} - [I(\boldsymbol{\hat{\theta}})]^{-1}$ it is straightforward to obtain a bound on the distance to the appropriate chisquare distribution for the score test. Due to space issues we do not pursue this application here. 
\end{enumerate} 
\end{remark}

\begin{proof}
Recall that by Assumption (R.C.5) the covariance matrix of $\boldsymbol{W}$ is $I(\boldsymbol{\theta_0})$, which is positive definite. In order to apply Theorem 2.{{3}} in \cite{Gaunt_Reinert} we need to show that the function $g$ in the assertion belongs to the class $C_{P}^2(\mathbb{R}^d)$. A function $g:\mathbb{R}^d\rightarrow\mathbb{R}$ is said to belong to the class $C_{P}^2(\mathbb{R}^d)$ if all second order partial derivatives of $g$ exist and there exists a dominating function $P:\mathbb{R}^d\rightarrow\mathbb{R}^+$ such that, for all $\boldsymbol{w}\in\mathbb{R}^d$, the partial derivatives satisfy
\begin{equation*}\bigg|\frac{\partial^kg(\boldsymbol{w})}{\prod_{j=1}^k\partial w_{i_j}}\bigg|^{2/k}\leq P(\mathbf{w}):=A+\sum_{i=1}^d B_i|w_i|^{r_i}, \quad k=1,2,
\end{equation*}
where $A\geq 0$, $B_1,\ldots,B_d \geq 0$ and $r_1,\ldots,r_d \geq 0$. 
The first partial derivatives of $g$ are 
\begin{align}
\nonumber & \frac{\partial}{\partial x_i} g( \boldsymbol{x},\boldsymbol{y}) =  2 \sum_{j=1}^{r}U_{j,i} \left( x_i-  \sum_{s=1}^{d-r} D_{i,s} y_s \right), \quad i=1, \ldots, r 
\\
\nonumber & \frac{\partial}{\partial y_i} g( \boldsymbol{x}, \boldsymbol{y}) 
=  - 2 \sum_{j=1}^r \sum_{l=1}^{r} D_{j,i} U_{j, l}  \left( x_l -  \sum_{s=1}^{d-r} D_{l,s} y_s \right), \quad i=1, \ldots, d-r.
\end{align}
Hence with $c$ given in \eqref{c},
\begin{equation}
\nonumber \bigg|\frac{\partial g(\boldsymbol{x}, \boldsymbol{y})}{\partial x_{i}}\bigg|^2 \leq  4  c^2 \left(|x_i| + \sum_{s=1}^{d-r} | y_s| \right)^2; \;\; 
 \bigg|\frac{\partial}{\partial y_i} g( \boldsymbol{x}, \boldsymbol{y}) \bigg|^2 \leq  4c^2  \left(\sum_{i=1}^r |x_i| + \sum_{s=1}^{d-r} | y_s| \right)^2.
\end{equation} 
So with $\boldsymbol{w} = (\boldsymbol{x}, \boldsymbol{y})$ the concatenation of $\boldsymbol{x}$ and $ \boldsymbol{y}$,
\begin{equation}
\nonumber \bigg|\frac{\partial g(\boldsymbol{w})}{\partial w_{i}}\bigg|^{2} \leq 4c^2 \left(\sum_{i=1}^d |w_i|  \right)^2 
\le 4dc^2 \sum_{i=1}^d |w_i|^2. 
\end{equation}
Similarly, for the second partial derivatives, 
with $c$ given in \eqref{c} and $\boldsymbol{w} = (\boldsymbol{x}, \boldsymbol{y})$ the concatenation of $\mathbf{x}$ and $ \mathbf{y}$,
\begin{equation}
\nonumber \bigg|\frac{\partial^2 g(\boldsymbol{w})}{\prod_{j=1}^2\partial w_{i_j}}\bigg| \leq 2  c.  
\end{equation}
All higher partial derivatives of $g$ vanish. 
Hence for $k=1, 2$ 
\begin{equation*}\bigg|\frac{\partial^k g(\boldsymbol{w})}{\prod_{j=1}^k\partial w_{i_j}}\bigg|^{2/k}
\leq A+\sum_{i=1}^d B_i|w_i|^{r_i}, \end{equation*}
with $A = 2c, B_i = 4c^2 d, r_i = 2$ for $i=1, \ldots, d$. Now applying Theorem {{2.3}}  in \cite{Gaunt_Reinert} gives that 
\begin{equation}
\nonumber \left|\mathbb{E}\left[h\left(g\left(\boldsymbol{W}\right)\right)\right]-\mathbb{E}\left[h\left(g\left(\left[I(\boldsymbol{\theta_0})\right]^{1/2}\mathbf{Z}\right)\right)\right]\right| \leq \frac{2 (||h' || + || h''||) }{\sqrt{n}} R(\boldsymbol{W}, U, D) 
\end{equation}
with $R(\boldsymbol{W}, U, D)$ as in \eqref{RWLD}, which gives the assertion. 
\end{proof} 

\subsubsection{Bounding the remainder terms}

In this section we shall bound expectations of the log-likelihood ratio statistics under smooth test functions.  Let $h \in \C_b^2(\mathbb{R})$ and $K \sim \chi^2_{r}$.
Using the triangle inequality and \eqref{cor:key}, 
\begin{align*}
\nonumber & \left|\mathbb{E}\left[h\left(-2\log \Lambda\right)\right]- \mathbb{E}[h(K)]\right| = \left|\mathbb{E}\left[h\left(g(\boldsymbol{\xi},\boldsymbol{\eta}) + R_{A_1} + R_{A_2} + R_{B_1} + R_{B_2}\right)\right]- \mathbb{E}[h(K)]\right|\\
&\leq \left|\mathbb{E}\left[h\left(g(\boldsymbol{\xi},\boldsymbol{\eta}) + R_{A_1} + R_{A_2} + R_{B_1} + R_{B_2}\right) - h\left(g(\boldsymbol{\xi},\boldsymbol{\eta}) + R_{B_1} + R_{B_2}\right)\right]\right| 
\\
& \;\; + \left|\mathbb{E}\left[h\left(g(\boldsymbol{\xi},\boldsymbol{\eta}) + R_{B_1} + R_{B_2}\right) - h\left(g(\boldsymbol{\xi},\boldsymbol{\eta})\right)\right]\right|+ \left|\mathbb{E}\left[h\left(g(\boldsymbol{\xi},\boldsymbol{\eta})\right)\right] - \mathbb{E}\left[h(K)\right]\right|.
\end{align*}
The terms to bound are hence
\begin{equation}
\left|\mathbb{E}\left[h\left(g(\boldsymbol{\xi},\boldsymbol{\eta}) + R_{A_1} + R_{A_2} + R_{B_1} + R_{B_2}\right) - h\left(g(\boldsymbol{\xi},\boldsymbol{\eta}) + R_{B_1} + R_{B_2}\right)\right]\right| \label{term1_multi}
\end{equation} 
and
\begin{equation}
\label{term2_multi}
 \left|\mathbb{E}\left[h\left(g(\boldsymbol{\xi},\boldsymbol{\eta}) + R_{B_1} + R_{B_2}\right) - h\left(g(\boldsymbol{\xi},\boldsymbol{\eta})\right)\right]\right|
\end{equation} 
as well as 
\begin{equation}
\label{term3_multi}
 \left|\mathbb{E}\left[h\left(g(\boldsymbol{\xi},\boldsymbol{\eta})\right)\right] - \mathbb{E}\left[h(K)\right]\right|.
\end{equation} 
We now bound these three terms  in order to complete the proof of Theorem \ref{Theorem_i.n.i.d}.
\vspace{0.05in}
\\
1.{\textbf{Bounding Term \eqref{term1_multi}}}\\ 
For $t(\boldsymbol{X})$ between $g(\boldsymbol{\xi},\boldsymbol{\eta}) + R_{A_1} + R_{A_2} + R_{B_1} + R_{B_2}$ and $g(\boldsymbol{\xi},\boldsymbol{\eta}) + R_{B_1} + R_{B_2}$, a first order Taylor expansion yields
\begin{equation}
\nonumber \eqref{term1_multi} = \left|\mathbb{E}\left[h'(t(\boldsymbol{X}))(R_{A_1} + R_{A_2})\right]\right| \leq \|h'\|\mathbb{E}\left[\left|R_{A_1}\right| + \left|R_{A_2}\right|\right].
\end{equation}
We start by bounding $\mathbb{E}\left|R_{A_1}\right|$, where
\begin{align}
\label{R_Amidbound}
\mathbb{E}\left|R_{A_1}\right| \leq & \mathbb{E}\left|R_1\right| + \mathbb{E}\left|\boldsymbol{R_3^{\intercal}}\boldsymbol{R_3}\right| + 2\sqrt{n}\mathbb{E}\left|\left(\boldsymbol{\hat{\theta}_n(X)} - \boldsymbol{\theta_0}\right)^{\intercal}[I(\boldsymbol{\theta_0})]^{\frac{1}{2}}\boldsymbol{R_3}\right|.
\end{align}
With the notation in \eqref{cm}, 
\begin{align}
\label{bound_for_R1}
\nonumber \mathbb{E}\left|R_1\right| & \leq \sum_{j=1}^d\sum_{k=1}^d\mathbb{E}\left|Q_jQ_kT_{kj}\right|\leq \sum_{j=1}^d\sum_{k=1}^d\left[\mathbb{E}\left(Q_j^2Q_k^2\right)\right]^{\frac{1}{2}}\left[\mathbb{E}\left(T_{kj}^2\right)\right]^{\frac{1}{2}}\\
& = \sqrt{n}\sum_{j=1}^d\sum_{k=1}^d\left[\mathbb{E}\left(Q_j^2Q_k^2\right)\right]^{\frac{1}{2}}\left[{\rm Var}\left(\frac{\partial^2}{\partial\theta_j\partial\theta_k}\log f(\boldsymbol{X_1}|\boldsymbol{\theta_0})\right)\right]^{\frac{1}{2}}.
\end{align}
Using H\"{o}lder's inequality twice, 
\begin{align}
\label{R3tR3boundfinal}
 & \mathbb{E}\left|\boldsymbol{R_3}^{\intercal}\boldsymbol{R_3}\right| \leq \frac{1}{n}\sum_{l=1}^{d}\sum_{m=1}^{d}\left[\left[I(\boldsymbol{\theta_0})\right]^{-1}\right]_{lm}\sum_{j=1}^{d}\sum_{k=1}^{d}\mathbb{E}\left|Q_jQ_kT_{lj}T_{mk}\right|\\
\nonumber & \leq \frac{1}{\sqrt{n}}\sum_{l=1}^{d}\sum_{m=1}^{d}\left[\left[I(\boldsymbol{\theta_0})\right]^{-1}\right]_{lm}\\
& \quad  \times \sum_{j=1}^{d}\sum_{k=1}^{d}\sqrt{{\rm Var}\left(\frac{\partial^2}{\partial\theta_l\partial\theta_j}\log f(X_1|\boldsymbol{\theta_0})\right)}\left[\mathbb{E}\left(Q_j^6\right)\mathbb{E}\left(Q_k^6\right)\mathbb{E}\left(T_{mk}^6\right)\right]^{\frac{1}{6}}. \nonumber
\end{align}

Moreover, 
\begin{align}
\label{bound3RA}
\nonumber & 2\sqrt{n}\mathbb{E}\left|\left(\boldsymbol{\hat{\theta}_n(X)} - \boldsymbol{\theta_0}\right)^{\intercal}[I(\boldsymbol{\theta_0})]^{\frac{1}{2}}\boldsymbol{R_3}\right| \leq 2\sum_{l=1}^d\sum_{j=1}^d\mathbb{E}\left|Q_lQ_jT_{lj}\right|\\
& \leq 2\sqrt{n}\sum_{l=1}^d\sum_{j=1}^d\sqrt{\mathbb{E}\left(Q_l^2Q_j^2\right)}\sqrt{{\rm Var}\left(\frac{\partial^2}{\partial\theta_l\partial\theta_j}\log f(X_1|\boldsymbol{\theta_0})\right)}.
\end{align}
Combining the results in \eqref{R_Amidbound}, \eqref{bound_for_R1}, \eqref{R3tR3boundfinal}, and \eqref{bound3RA}, yields to
\begin{align}
\nonumber & \mathbb{E}\left|R_{A_1}\right| \leq \frac{1}{\sqrt{n}}K_1(\boldsymbol{\theta_0}),
\end{align}
with $K_1(\boldsymbol{\theta_0})$ as in \eqref{K1}. In order to bound $\mathbb{E}\left|R_{A_2}\right|$, we follow exactly the same process that was followed to bound $\mathbb{E}\left|R_{A_1}\right|$, but now under the null hypothesis,
to conclude that
$\mathbb{E}\left|R_{A_2}\right| \leq \frac{1}{\sqrt{n}}K_1^*(\boldsymbol{\theta_0}),
$
and therefore
\begin{align}
\label{boundsecondterm}
& \eqref{term1_multi} \leq \frac{1}{\sqrt{n}}\left(K_1(\boldsymbol{\theta_0}) + K_1^*(\boldsymbol{\theta_0})\right).
\end{align}
2.{\textbf{Bounding Term \eqref{term2_multi}}}\\
The terms in $R_{B_1}$ and $R_{B_2}$ of \eqref{RARBterms} may not be uniformly bounded in $\boldsymbol{\theta}$. The triangle inequality leads to
\begin{align}
\label{A_3_1}
\eqref{term2_multi} & \leq \left|\mathbb{E}\left[h\left(g(\boldsymbol{\xi},\boldsymbol{\eta}) + R_{B_1} + R_{B_2}\right) - h\left(g(\boldsymbol{\xi},\boldsymbol{\eta}) + R_{B_2}\right)\right]\right|
\\
\label{A_3_2}
& \qquad + \left|\mathbb{E}\left[h\left(g(\boldsymbol{\xi},\boldsymbol{\eta}) + R_{B_2}\right) - h\left(g(\boldsymbol{\xi},\boldsymbol{\eta})\right)\right]\right|.
\end{align}
{{\textit{Bound for} \eqref{A_3_1}:}} Let $0 < \epsilon \le  \epsilon(\boldsymbol{\theta_0})$. With $Q_{(m)}$ as in \eqref{cm}, the law of total expectation, the Cauchy-Schwarz inequality, and Markov's inequality yield
\begin{align*}
\nonumber & \eqref{A_3_1} \leq  \mathbb{E}\left|h\left(g(\boldsymbol{\xi},\boldsymbol{\eta}) + R_{B_1} + R_{B_2}\right) - h\left(g(\boldsymbol{\xi},\boldsymbol{\eta}) + R_{B_2}\right)\right|\\
\nonumber & \leq 2\|h\|\Prob\left(\left|Q_{{m}}\right|\geq \epsilon\right) \\
\nonumber & +\mathbb{E}\left[\left|h\left(g(\boldsymbol{\xi},\boldsymbol{\eta}) + R_{B_1} + R_{B_2}\right) - h\left(g(\boldsymbol{\xi},\boldsymbol{\eta}) + R_{B_2}\right)\right|\middle| \left|Q_{(m)}\right| < \epsilon\right]\Prob\left(\left|Q_{(m)}\right|<\epsilon\right)\\
\nonumber & \leq 2\frac{\|h\|}{\epsilon^2}\mathbb{E}\left(\sum_{j=1}^{d}Q_j^2\right) \\
& \qquad\;\;+ \mathbb{E}\left[\left|h\left(g(\boldsymbol{\xi},\boldsymbol{\eta}) + R_{B_1} + R_{B_2}\right) - h\left(g(\boldsymbol{\xi},\boldsymbol{\eta}) + R_{B_2}\right)\right|\middle| \left|Q_{(m)}\right| < \epsilon\right].
\end{align*}
A first order Taylor expansion yields
\begin{align}
\label{midstepforA_3_1}
\nonumber & \mathbb{E}\left[\left|h\left(g(\boldsymbol{\xi},\boldsymbol{\eta}) + R_{B_1} + R_{B_2}\right) - h\left(g(\boldsymbol{\xi},\boldsymbol{\eta}) + R_{B_2}\right)\right|\middle| \left|Q_{(m)}\right| < \epsilon\right]\\
\nonumber & \leq \|h'\|\mathbb{E}\left[\left|R_{B_1}\right|\middle|\left|Q_{(m)}\right|<\epsilon\right]\\
\nonumber & \leq \|h'\|\mathbb{E}\left[\left(\left|R_2\right|+ \left|\boldsymbol{R_4}^{\intercal}\left(\boldsymbol{R_3} + \boldsymbol{R_4}\right)\right|+\left|\boldsymbol{R_3}^{\intercal}\boldsymbol{R_4}\right|\right.\right.\\
& \left.\left.\qquad\qquad\; + 2\sqrt{n}\left|\left(\boldsymbol{\hat{\theta}_n(X)} - \boldsymbol{\theta_0}\right)^{\intercal}\left[I(\boldsymbol{\theta_0})\right]^{\frac{1}{2}}\boldsymbol{R_4}\right|\right)\middle|\left|Q_{(m)}\right|<\epsilon\right].
\end{align}
From now on, we denote
\begin{equation}
\label{Delta_qml}
\Delta_{qsl}:=\Delta_{qsl}\left(\boldsymbol{X},\boldsymbol{\theta_0}\right) = \frac{\partial^3}{\partial\theta_q\partial\theta_s\partial\theta_l}\ell(\boldsymbol{\theta_0^*};\boldsymbol{X})
\end{equation}
and we bound the terms in \eqref{midstepforA_3_1} in turns.\\ 

{\textit{Bound for }}$\mathbb{E}\left|R_2\right|$: With $R_2$ as in Proposition \ref{t1}, it is straightforward that for $\Delta_{qsl}$ as in \eqref{Delta_qml},
\begin{align}
\label{newboundR2}
\nonumber \mathbb{E}\left(\left|R_2\right|\middle|\left|Q_{(m)}\right|<\epsilon\right) & \leq \frac{4}{3}\sum_{j=1}^d\sum_{k=1}^d\sum_{s=1}^d\mathbb{E}\left(\left|Q_jQ_kQ_s\Delta_{jks}\right|\middle|\left|Q_{(m)}\right|<\epsilon\right)\\
& \leq \frac{4}{3}\sum_{j=1}^{d}\sum_{k=1}^{d}\sum_{s=1}^{d}\sqrt{\mathbb{E}\left(Q_j^2Q_k^2Q_s^2\right)}\left[\mathbb{E}\left(\left(M_{jks}(\boldsymbol{X})\right)^2\middle|\left|Q_{(m)}\right|<\epsilon\right)\right]^{\frac{1}{2}}. 
\end{align}

{\textit{Bound for }}$\mathbb{E}\left(\left|\boldsymbol{R_4}^{\intercal}\boldsymbol{R_3}\right|\middle|\left|Q_{(m)}\right|<\epsilon\right)$ {\textit{and}} $\mathbb{E}\left(\left|\boldsymbol{R_3}^{\intercal}\boldsymbol{R_4}\right|\middle|\left|Q_{(m)}\right|<\epsilon\right)$: With $\Delta_{qsl}$ as in \eqref{Delta_qml} and $T_{kj}$ as in \eqref{cm}, using H\"{o}lder's inequality and \cite{anastasiou2015assessing}, Lemma 4.1, we obtain that
\begin{align}
\label{boundR4tR3}
 & \mathbb{E}\left(\left|\boldsymbol{R_4}^{\intercal}\boldsymbol{R_3}\right|\middle|\left|Q_{(m)}\right|<\epsilon\right)\\
\nonumber & \leq \frac{1}{2n}\sum_{q=1}^{d}\sum_{k=1}^{d}\left|\left[\left[I(\boldsymbol{\theta_0})\right]^{-1}\right]_{kq}\right|\sum_{j=1}^{d}\sum_{l=1}^{d}\sum_{s=1}^{d}\mathbb{E}\left[\left|Q_jQ_lQ_sT_{kj}\Delta_{qsl}\right|\middle|\left|Q_{(m)}\right|<\epsilon\right]\\
\nonumber & \leq \frac{1}{2n}\sum_{q=1}^{d}\sum_{k=1}^{d}\left|\left[\left[I(\boldsymbol{\theta_0})\right]^{-1}\right]_{kq}\right|\sum_{j=1}^{d}\sum_{l=1}^{d}\sum_{s=1}^{d}\sqrt{\mathbb{E}\left(Q_j^2Q_l^2Q_s^2\right)}\left[\mathbb{E}\left(T_{kj}^4\middle|\left|Q_{(m)}\right|<\epsilon\right)\right]^{\frac{1}{4}}\\
&\qquad\quad\times\left[\mathbb{E}\left(\left(M_{qml}(\boldsymbol{X})\right)^4\middle|\left|Q_{(m)}\right|<\epsilon\right)\right]^{\frac{1}{4}}.\nonumber
\end{align}
Since $\boldsymbol{R_3}^{\intercal}\boldsymbol{R_4} = \boldsymbol{R_4}^{\intercal}\boldsymbol{R_3}$,  $\mathbb{E}\left(\left|\boldsymbol{R_3}^{\intercal}\boldsymbol{R_4}\right|\middle|\left|Q_{(m)}\right|<\epsilon\right)$ can also be bounded by  \eqref{boundR4tR3}.\\
{\textit{Bound for }}$\mathbb{E}\left(\left|\boldsymbol{R_4}^{\intercal}\boldsymbol{R_4}\right|\middle|\left|Q_{(m)}\right|<\epsilon\right)$: 
Again with \cite{anastasiou2015assessing}, Lemma 4.1 and the Cauchy-Schwarz inequality, and with $\Delta_{qsl}$ as in \eqref{Delta_qml},
\begin{align}
\label{boundR4tR4}
& \mathbb{E}\left(\left|\boldsymbol{R_4}^{\intercal}\boldsymbol{R_4}\right|\middle|\left|Q_{(m)}\right|<\epsilon\right)\\
\nonumber & \leq \frac{1}{4n}\sum_{b=1}^{d}\sum_{k=1}^{d}\sum_{s=1}^{d}\sum_{q=1}^{d}\sum_{l=1}^{d}\sum_{j=1}^{d}\left|\left[\left[I(\boldsymbol{\theta_0})\right]^{-1}\right]_{qb}\right|\mathbb{E}\left[\left|Q_kQ_sQ_jQ_l\Delta_{bsk}\Delta_{qjl}\right|\middle|\left|Q_{(m)}\right|<\epsilon\right] \\
\nonumber & \leq \frac{1}{4n}\sum_{b=1}^{d}\sum_{k=1}^{d}\sum_{s=1}^{d}\sum_{q=1}^{d}\sum_{l=1}^{d}\sum_{j=1}^{d}\left|\left[\left[I(\boldsymbol{\theta_0})\right]^{-1}\right]_{qb}\right|\sqrt{\mathbb{E}\left(Q_k^2Q_s^2Q_j^2Q_l^2\right)}\\
& \qquad\quad \times \left[\mathbb{E}\left(\left(M_{bsk}(\boldsymbol{X})\right)^4\middle|\left|Q_{(m)}\right|<\epsilon\right)\right]^{\frac{1}{4}}\left[\mathbb{E}\left(\left(M_{qjl}(\boldsymbol{X})\right)^4\middle|\left|Q_{(m)}\right|<\epsilon\right)\right]^{\frac{1}{4}}. \nonumber 
\end{align}

{\textit{Bound for }}$\mathbb{E}\left(2\sqrt{n}\left|\left(\boldsymbol{\hat{\theta}_n(X)} - \boldsymbol{\theta_0}\right)^{\intercal}\left[I(\boldsymbol{\theta_0})\right]^{\frac{1}{2}}\boldsymbol{R_4}\right|\middle|\left|Q_{(m)}\right|<\epsilon\right)$: 
A similar process as the one  to obtain the bounds in \eqref{boundR4tR3} and \eqref{boundR4tR4} yields
\begin{align}
\label{bound_for_last_term_RB1}
& \mathbb{E}\left(2\sqrt{n}\left|\left(\boldsymbol{\hat{\theta}_n(X)} - \boldsymbol{\theta_0}\right)^{\intercal}\left[I(\boldsymbol{\theta_0})\right]^{\frac{1}{2}}\boldsymbol{R_4}\right|\middle|\left|Q_{(m)}\right|<\epsilon\right)\\
\nonumber & \leq \sum_{l=1}^{d}\sum_{j=1}^{d}\sum_{q=1}^{d}\mathbb{E}\left(\left|Q_lQ_jQ_q\Delta_{ljq}\right|\middle|\left|Q_{(m)}\right|<\epsilon\right)\\
& \leq \sum_{l=1}^{d}\sum_{j=1}^{d}\sum_{q=1}^{d}\sqrt{\mathbb{E}\left(Q_l^2Q_j^2Q_q^2\right)}\left[\mathbb{E}\left(\left(M_{ljq}(\boldsymbol{X})\right)^2\middle|\left|Q_{(m)}\right|<\epsilon\right)\right]^{\frac{1}{2}}.\nonumber 
\end{align}
Combining the results for the bound  on \eqref{A_3_1}, \eqref{midstepforA_3_1}, \eqref{boundR4tR3}, \eqref{boundR4tR4}, and \eqref{bound_for_last_term_RB1}, we conclude that
\begin{align}
\label{final_boundRB1}
\eqref{A_3_1} \leq & \frac{1}{\sqrt{n}}K_2(\boldsymbol{\theta_0}),
\end{align}
with $K_2(\boldsymbol{\theta_0})$ as in \eqref{K2}.\\

{\textit{Bound for} \eqref{A_3_2}:} With $Q^*_{(m)}$ as in \eqref{cm}, the law of total expectation, the Cauchy-Schwarz inequality, and Markov's inequality yield
\begin{align}
\nonumber \eqref{A_3_2} & \leq  \mathbb{E}\left|h\left(g(\boldsymbol{\xi},\boldsymbol{\eta}) + R_{B_2}\right) - h\left(g(\boldsymbol{\xi},\boldsymbol{\eta})\right)\right|\\
\nonumber & \leq 2\|h\|\Prob\left(\left|Q^*_{{m}}\right|\geq \epsilon\right)\\
\nonumber & \;\; + \mathbb{E}\left[\left|h\left(g(\boldsymbol{\xi},\boldsymbol{\eta}) + R_{B_2}\right) - h\left(g(\boldsymbol{\xi},\boldsymbol{\eta})\right)\right|\middle| \left|Q^*_{(m)}\right| < \epsilon\right]\Prob\left(\left|Q^*_{(m)}\right|<\epsilon\right).
\end{align}
Finding an upper bound for this expression follows the same arguments as the one 
to bound \eqref{A_3_1} and therefore, it will not be repeated; the result is 
\begin{align}
\label{final_boundRB2}
\eqref{A_3_2} \leq \frac{1}{\sqrt{n}}K_2^*(\boldsymbol{\theta_0}),
\end{align}
where $K_2^*(\boldsymbol{\theta_0})$ is the version of $K_2(\boldsymbol{\theta_0})$  under the null hypothesis.
\vspace{0.05in}
\\
3.{\textbf{Bounding Term \eqref{term3_multi}}}\\
From \eqref{lem:simple}, 
\begin{equation} 
\nonumber g(\boldsymbol{\xi},\boldsymbol{\eta}) = \left(\boldsymbol{\xi}- B C^{-1} \boldsymbol{\eta}\right)^{\intercal} \left( A - B C^{-1} B^{\intercal}\right)^{-1}\left(\boldsymbol{\xi}- B C^{-1}\boldsymbol{\eta}\right)
\end{equation}
and 
\begin{align}
\nonumber & \mathbb{E}\left[h\left(g\left(\boldsymbol{\xi},\boldsymbol{\eta}\right)\right)\right] - \mathbb{E}[h(K)] = \mathbb{E}\left[h\left(g\left(\boldsymbol{\xi},\boldsymbol{\eta}\right)\right)\right]  -\mathbb{E}\left[h\left(g\left([I(\boldsymbol{\theta_0})]^{1/2}\boldsymbol{Z}\right)\right)\right].
\end{align}
With
\begin{equation}
\nonumber U = \left( A - B C^{-1} B^T\right)^{-1} ; \quad D = B C^{-1}  , 
\end{equation}
Proposition \ref{gauntreinertcor} yields 
\begin{equation}
\nonumber \left| \mathbb{E}\left[h\left(g(\boldsymbol{\xi},\boldsymbol{\eta})\right)\right]  -\mathbb{E}\left[h\left(g\left(\left[I(\boldsymbol{\theta_0})\right]^{1/2}\mathbf{Z}\right)\right)\right]\right| \le 2  \frac{(||h' || + || h''||)}{\sqrt{n}} R\left(  \left(\boldsymbol{\xi},\boldsymbol{\eta}\right); U, D\right) 
\end{equation}
where $R\left(  \left(\boldsymbol{\xi},\boldsymbol{\eta}\right); U, D\right)$ is as in Proposition \ref{gauntreinertcor}. 

To see that $\mathbb{E}\left[h(K)\right]$ is indeed $\mathbb{E}\left[h\left(g\left(\left[I(\boldsymbol{\theta_0})\right]^{1/2}\boldsymbol{Z}\right)\right)\right]$,
let 
$$\boldsymbol{N}  = \left(A-BC^{-1} B^{\intercal}\right)^{1/2} \boldsymbol{Z} \in \RR^r$$
so that 
$\boldsymbol{N}\sim {\mathcal{MVN}}_r\left(\mathbf{0}, A-BC^{-1} B^{\intercal}\right)$, and 
 $\boldsymbol{N}^{\intercal}\left(A-BC^{-1} B^{\intercal}\right)^{-1}\boldsymbol{N} \sim \chi^2_r$, as 
 $A-BC^{-1} B^{\intercal}$ is positive definite. 
Moreover we can create $\boldsymbol{N}$ through the following construction (see for example Theorem 3.2.3 in \cite{mardia1979multivariate}). Let $\boldsymbol{M}\sim {\mathcal{MVN}}_d(\boldsymbol{0}, I(\boldsymbol{\theta_0}) )$, and decompose $\boldsymbol{M} = \left(\boldsymbol{M}^{1:r}, \boldsymbol{M}^{r+1:d}\right)$, then 
$ \boldsymbol{M}^{1:r}  - B C^{-1} \boldsymbol{M}^{r+1:d} \sim {\mathcal{MVN}}_r\left(\boldsymbol{0}, A-BC^{-1} B^{\intercal}\right).
$
Thus $\mathbb{E}\left[h(K)\right] = \mathbb{E}\left[h\left( g\left(\boldsymbol{M}\right)\right)\right].$ 

{\raggedright{Hence we conclude that for \eqref{term3_multi}}}
\begin{equation}
\label{gesinebound}
\left|\mathbb{E}\left[h\left(g(\boldsymbol{\xi},\boldsymbol{\eta}) \right) - h\left(K \right)\right]\right|\leq 2 \frac{(||h' || + || h''||)}{\sqrt{n}} R\left(  \left(\boldsymbol{\xi},\boldsymbol{\eta}\right); U, D \right) .
\end{equation} 
The results in \eqref{gesinebound}, \eqref{boundsecondterm}, \eqref{final_boundRB1} and \eqref{final_boundRB2} conclude the proof of Theorem \ref{Theorem_i.n.i.d}. $\hfill \blacksquare$ 

\medskip 
The next section gives three examples to illustrate the approach. Firstly we consider an example with a one-dimensional parameter, namely the exponential distribution. The second example is that of the normal distribution with two-dimensional parameter $(\mu, \sigma^2)$. The last example is logistic regression. 


\section{Examples} \label{sec:examples}

\subsection{Single-parameter-case example: the exponential distribution} 
Here, we apply Theorem \ref{Theorem_i.n.i.d} in an example from a single-parameter distribution. We highlight that in the single-parameter case the interest is on assessing the asymptotic $\chi^2_1$ distribution of $2\left(l\left(\hat{\theta}_n(\boldsymbol{X});\boldsymbol{X}\right) - l(\theta_0;\boldsymbol{X})\right)$, where $\theta_0$ is the true value of the unknown parameter $\theta$. The log-likelihood ratio in \eqref{loglikelihoodsplitT1T2} reduces to
$
\nonumber -2\log \Lambda = 2\log T_1,
$
so  that there is no need to introduce $T_2$ as defined in \eqref{loglikelihoodsplitT1T2} and the terms $K_1^*(\theta_0)$ and $K_2^*(\theta_0)$ in the expression of \eqref{final_bound_regression} vanish.

To illustrate the single-parameter case, we consider an example from the exponential distribution with mean $\theta_0$. For $X \sim {\rm Exp}\left(\frac{1}{\theta}\right)$, $\theta > 0$ the p.d.f. is
$
f(x|\theta) = \frac{1}{\theta}{\rm exp}\left\lbrace-\frac{1}{\theta}x\right\rbrace, $ for $ x>0.
$
\begin{corollary}
\label{Corollarynoncanexp}
Let $X_1, X_2, \cdots, X_n$ be i.i.d. random variables that follow the Exp$\left(\frac{1}{\theta_0}\right)$ distribution. The MLE exists, it is unique, equal to $\hat{\theta}_n(\boldsymbol{X}) = \bar{X}$ and the regularity conditions (R.C.1)-(R.C.5) as well as (O2) and (O3) are satisfied. For   $h \in \C_b^2(\mathbb{R})$ and  $K\sim \chi_1^2$, we have that
\begin{align}
\label{boundexponential}
\nonumber & \left|\mathbb{E}\left[h\left(2\left(\ell\left(\hat{\theta}_n(\boldsymbol{X});\boldsymbol{X}\right) - \ell(\theta_0;\boldsymbol{X})\right)\right)\right] - \mathbb{E}\left[h\left(K\right)\right]\right|< 8\frac{\|h\|}{n}\\
\nonumber & \;\; +  \frac{\sqrt{2}}{\theta_0^8\sqrt{\pi}}\left(19\theta_0^4 + 325\theta_0^2 + 2733 + \frac{36973}{n}\right)\\
\nonumber & \;\; + \frac{\|h'\|}{\sqrt{n}}\left\lbrace\vphantom{(\left(\sup_{\theta:|\theta-\theta_0|\leq\epsilon}\left|l^{(3)}(\theta;\boldsymbol{X})\right|\right)^2}6\sqrt{3+\frac{6}{n}}+\sqrt{15+\frac{130}{n}+\frac{120}{n^2}}\left(\frac{1120}{3}+\frac{320\left(3+\frac{6}{n}\right)^{\frac{1}{4}}+4}{\sqrt{n}}\right)\right.\\
& \qquad\qquad\quad \left. + \frac{6400}{\sqrt{n}}\sqrt{105+\frac{2380}{n}+ \frac{7308}{n^2} + \frac{5040}{n^3}}\vphantom{(\left(\sup_{\theta:|\theta-\theta_0|\leq\epsilon}\left|l^{(3)}(\theta;\boldsymbol{X})\right|\right)^2}\right\rbrace.
\end{align}
\end{corollary}
\begin{remark}
\textbf{(1)} The upper bound in \eqref{boundexponential} is $\mathcal{O}\left(\frac{1}{\sqrt{n}}\right)$.\\
\textbf{(2)} The normal approximation should be poor when $\theta_0$ is close to zero since the variance is then very large. This is reflected in our bound, which is small only when $\theta_0$ is such that $\sqrt{n}\theta_0^8$ is large.
\end{remark}
\begin{proof}
It is easy to check that the assumptions of Theorem \ref{Theorem_i.n.i.d} hold. Here we choose 
$\epsilon(\theta_0) = \frac12 \theta_0$ and with $\bar{x} = \frac1n \sum_{i=1}^n x_i$, 
$$M_{i,j,k} (\boldsymbol{x}) = M_{1,1,1}  (\boldsymbol{x}) = \frac{96}{\theta_0^4} \sum_{i=1}^n x_i + \frac{16}{\theta_0^3} = \frac{96 n}{\theta^4} \left( 3 {\bar{x}} + \frac12 \theta_0 \right).$$
In addition, straightforward calculations lead to $\hat{\theta}_n(\boldsymbol{X})  = \bar{X}$.
The expected Fisher information number for one random variable is
$
i(\theta_0)
= \frac{1}{\theta_0^2}.$ 
We start with the calculation of the first term of the bound in \eqref{final_bound_regression}. Since $d=1, r=0$, the vector $\boldsymbol{W} = (\boldsymbol{\xi},\boldsymbol{\eta})$ reduces to 
$ \eta = \frac{1}{\sqrt{n}}\sum_{j=1}^{n}Y_{j},
$
where $Y_j = \frac{{\rm d}}{{\rm d}\theta}\log f(X_j|\theta_0) = \frac{X_j-\theta_0}{\theta_0^2}$. From the definition of $U$, $D$,  for $c$ as in \eqref{c},
$ c = \max\left\lbrace \frac{1}{\theta_0^2}, 0 \right\rbrace = \frac{1}{\theta_0^2}.
$
Therefore, for $Z \sim {\rm N}(0,1)$ the aim is to bound
\begin{align}
\nonumber & R(\eta, U, D) = \frac{1}{n\theta_0^2}\mathbb{E}\left|\theta_0 Z\right|\sum_{i=1}^n\left\lbrace\vphantom{(\left(\sup_{\theta:|\theta-\theta_0|\leq\epsilon}\left|l^{(3)}(\theta;\boldsymbol{X})\right|\right)^2}\mathbb{E}|Y_{i}^3| + \frac{8}{\theta_0^2} \left(4 \mathbb{E}|Y_{i}^3|\mathbb{E}\left(\eta^2\right) +\frac{4}{n}\mathbb{E}|Y_{i}^5|+\frac{\mathbb{E}\left|\theta_0Z^3\right|}{\mathbb{E}\left|\theta_0Z\right|}\mathbb{E}|Y_{i}^3|\right)\right.\\
\nonumber & \qquad\qquad\qquad \left. +2\left|\mathbb{E}\left(Y_{i}^2\right)\right|\left[\mathbb{E}|Y_{i}| + \frac{16}{\theta_0^2}\left(4\mathbb{E}|Y_{i}|\mathbb{E}\left(\eta^2\right)+\frac{4}{n}\mathbb{E}|Y_{i}^{3}| +\frac{\mathbb{E}\left|\theta_0Z^{3}\right|}{\mathbb{E}\left|\theta_0Z\right|}\mathbb{E}|Y_{i}|\right)\right]\vphantom{(\left(\sup_{\theta:|\theta-\theta_0|\leq\epsilon}\left|l^{(3)}(\theta;\boldsymbol{X})\right|\right)^2}\right\rbrace.
\end{align}
With
\begin{align}
\nonumber & \mathbb{E}\left|Z\right| = \sqrt{\frac{2}{\pi}}, \quad \mathbb{E}\left|Z^3\right| = 2\sqrt{\frac{2}{\pi}}, \quad \EE\left|Y_i\right| \leq \frac{1}{\theta_0}, \quad \EE\left(Y_i^2\right) = \frac{1}{\theta_0^2}\\
\nonumber & \EE\left|Y_i^3\right| \leq \frac{\sqrt{265}}{\theta_0^3}, \quad \EE\left|Y_i^5\right| \leq \frac{\sqrt{1334961}}{\theta_0^5},\quad \EE\left(\eta^2\right) = \frac{1}{\theta_0^2},
\end{align}
then
\begin{equation}
\label{bound_R_exponential}
R(\eta, U, D) < \frac{\sqrt{2}}{\theta_0^8\sqrt{\pi}}\left(19\theta_0^4 + 325\theta_0^2 + 2733 + \frac{36973}{n}\right).
\end{equation}
The next task is to bound $K_1(\theta_0)$ as in \eqref{K1}, for $d=1$. Using the definition of $Q_j$ in \eqref{cm},  $Q_1 = \bar{X} - \theta_0$. The moments of $Q_1$ are calculated using standard results from 
\cite{DistributionTheory} along with the fact that $\bar{X} \sim G\left(n, \frac{n}{\theta_0}\right)$, giving  
\begin{align}
\label{bound_first_term_K1_exponential}
3n\sqrt{\EE\left(Q_1\right)^4}\left[{\rm Var}\left(\frac{{\rm d}^2}{{\rm d}\theta^2}\log f(X_1|\theta_0)\right)\right]^{\frac{1}{2}} & = 
 6\sqrt{3+\frac{6}{n}}.
\end{align}
For the second quantity in \eqref{K1}, with the definition of $T_{11}$ in \eqref{cm},
\begin{align}
\label{bound_second_term_K1_exponential}
\nonumber & \frac{1}{i(\theta_0)}\left[{\rm Var}\left(\frac{{\rm d}^2}{{\rm d}\theta^2}\log f(X_1|\theta_0)\right)\right]^{\frac{1}{2}}\left[\mathbb{E}\left(Q^6\right)\right]^{\frac{1}{3}}\left[\mathbb{E}\left(T_{11}^6\right)\right]^{\frac{1}{6}}\\
\nonumber & = 2\left[\EE\left(\bar{X}-\theta_0\right)^6\right]^{\frac{1}{3}}\left[\EE\left(-\frac{2n\bar{X}}{\theta_0^3} + \frac{2n}{\theta_0^2}\right)^6\right]^{\frac{1}{6}}\\
& = \frac{4n}{\theta_0^3}\sqrt{\EE\left(\bar{X}-\theta_0\right)^6} = \frac{4}{\sqrt{n}}\sqrt{15+\frac{130}{n}+\frac{120}{n^2}}.
\end{align}
Combining  \eqref{bound_first_term_K1_exponential} and \eqref{bound_second_term_K1_exponential}, 
\begin{equation}
\label{bound_K1_exponential}
K_1(\theta_0) = 6\sqrt{3+\frac{6}{n}} + \frac{4}{\sqrt{n}}\sqrt{15+\frac{130}{n}+\frac{120}{n^2}}.
\end{equation}
We proceed to find a bound for $K_2(\boldsymbol{\theta_0})$, as defined in \eqref{K2}. The calculation of the first term is straightforward;
\begin{equation}
\label{bound_first_term_K2_exponential}
2\sqrt{n}\frac{\|h\|}{\epsilon^2}\mathbb{E}\left(\bar{X}-\theta_0\right)^2 = \frac{2\|h\|\theta_0^2}{\sqrt{n}\epsilon^2}.
\end{equation}
The second term of \eqref{K2} requires the calculation of conditional expectations related to $M_{111}(\boldsymbol{X})$. For $ \epsilon = \frac12 \theta_0$, 
\begin{align}
\label{bound_second_term_K2_exponential}
 & \sqrt{n}\|h'\|\frac{7}{3}\sqrt{\mathbb{E}\left(Q_1^6\right)}\left[\mathbb{E}\left[\left(M_{111}(\boldsymbol{X})\right)^2\middle|\left|Q_{1}\right|<\epsilon\right]\right]^{\frac{1}{2}}\\
\nonumber & = \sqrt{n}\|h'\|\frac{7}{3}\sqrt{\mathbb{E}\left(\bar{X}-\theta_0\right)^6}\left[\mathbb{E}\left[ \frac{96^2 n^2}{\theta^8} \left( 3 {\bar{X}} + \frac12 \theta_0 \right)^2\middle|\left|\bar{X}-\theta_0\right|<\frac12 \theta_0\right]\right]^{\frac{1}{2}}\\
\nonumber & < \frac{448\|h'\|\theta_0^3}{3 \theta_0^4}\sqrt{15+\frac{130}{n}+\frac{120}{n^2} }\left(2\theta_0+\frac12 \theta_0 \right)\\
& = \frac{1120}{3} \|h'\| \sqrt{15+\frac{130}{n}+\frac{120}{n^2} }. \nonumber
\end{align}
Bounding the third term of \eqref{K2} requires the calculation of conditional expectations related to $T_{11}$ of \eqref{cm} and $M_{111}(\boldsymbol{X})$. It is easy to see that $T_{11}$ can be written as a continuous, increasing function of $Q_1$. Therefore, employing Lemma 2.1 of \cite{Anastasiou_Reinert}, leads to
\begin{align}
\nonumber & \frac{\|h'\|}{\sqrt{n}}\frac{1}{i(\theta_0)}\sqrt{\mathbb{E}\left(Q_1^6\right)}\left[\mathbb{E}\left(T_{11}^4\middle|\left|Q_{1}\right|<\frac12 \theta_0\right)\right]^{\frac{1}{4}}\left[\mathbb{E}\left(\left(M_{111}(\boldsymbol{X})\right)^4\middle|\left|Q_{1}\right|<\frac12 \theta_0\right)\right]^{\frac{1}{4}}\\
\nonumber & < \frac{128 n^{\frac{3}{2}}\|h'\|}{\theta_0^5 }\sqrt{\frac{\theta_0^6}{n^3}\left(15+\frac{130}{n}+\frac{120}{n^2}\right)}\left[\mathbb{E}\left(\bar{X}-\theta_0\right)^4\right]^{\frac{1}{4}}\frac{5}{2} \theta_0  .
\end{align}
An upper bound for the fourth term of \eqref{K2} is found in a similar way. 
Collecting these bounds gives 
\begin{align}
\label{bound_K2_exponential}
\nonumber K_2(\theta_0) & < \frac{8\|h\|}{\sqrt{n}} + \|h'\|\sqrt{15+\frac{130}{n}+\frac{120}{n^2}}\left(\frac{1120}{3} + \frac{320(3)^{\frac{1}{4}}}{\sqrt{n}}\left(\frac{2}{n}+1\right)^{\frac{1}{4}}\right)\\
& \quad + \frac{6400\|h'\|}{\sqrt{n}}\sqrt{105+\frac{2380}{n}+\frac{7308}{n^2}+\frac{5040}{n^3}}
\end{align} 
Combining now the results in \eqref{bound_R_exponential}, \eqref{bound_K1_exponential}, and \eqref{bound_K2_exponential} yields the assertion.
\end{proof}

\begin{remark}
We chose  $\epsilon(\theta_0)$ 
to be the mid-point of the interval $(0,\theta_0)$ as there is a trade off on its choice for $K_2(\boldsymbol{\theta_0})$. A more systematic choice of $\epsilon(\theta_0)$  based on numerical solutions of inequalities could be of interest in principle. As our bounds are not optimised with respect to the constants, for space reasons this systematic choice is not carried out. 

Here is a numerical example of the behaviour of the bound in \eqref{boundexponential}, using a specific function and a specific value for the parameter $\theta_0$. The function is $h_t(x) = \frac{1}{x^2+2}$, and from simple calculations,
\begin{equation}
\nonumber \|h_t\| \leq \frac{1}{2}, \quad \|h'_t\| \leq \frac{3\sqrt{1.5}}{16}, \quad \|h''_t\| \leq \frac{1}{2},
\end{equation}
meaning that $h_t \in \C_b(\mathbb{R})$. Taking $\theta_0 = 3$ and  $n=10^5$,  the bound is equal to 1.216. It is instructive to examine the contributions to this bound: 
\begin{align}
\nonumber & 2  \frac{(||h_t' || + || h_t''||)}{\sqrt{n}} R\left(  \left(\boldsymbol{\xi},\boldsymbol{\eta}\right), (A - B C^{-1} B^T)^{-1}, B C^{-1} \right) < 0.004\\
\nonumber & \frac{\|h'_t\|}{\sqrt{n}}K_1(\theta_0) = 0.008\\
\nonumber & \frac{1}{\sqrt{n}}K_2(\theta_0) < 1.204.
\end{align}
The bound is heavily dependent on the quantity related to $K_2(\theta_0)$, whereas the values for $K_1(\theta_0)$ and $R\left(  \left(\boldsymbol{\xi},\boldsymbol{\eta}\right), (A - B C^{-1} B^T)^{-1}, B C^{-1} \right)$ are very small. This arises from  the large and non-optimised constants in $K_2(\theta_0)$  of \eqref{bound_K2_exponential}.
\end{remark}

\subsection{Example: the normal distribution} 
Here, we apply Theorem \ref{Theorem_i.n.i.d} in the case of $X_1, X_2, \ldots, X_n$ i.i.d.  random variables from N$(\mu,\sigma^2)$ with $\boldsymbol{\theta} = (\mu,\sigma^2) \in \RR \times \RR^+$. We consider the test problem $H_0: \mu=0$ against the general alternative. It is well-known that under the alternative, the MLE is equal to $\boldsymbol{\hat{\theta}_n(X)} = \left(\hat{\mu}, \hat{\sigma}^2\right)^{\intercal} = \left(\bar{X}, \frac{1}{n}\sum_{i=1}^{n}(X_i-\bar{X})^2\right)^{\intercal}$; see for example \cite{Davison}, p.116. Under the null, simple calculations show that the MLE for $\sigma^2$ is $\hat{\theta}_{*}(\boldsymbol{X}) = \frac{1}{n}\sum_{i=1}^{n}X_i^2$. In addition, the regularity conditions (R.C.1)-(R.C.5) are satisfied.
\vspace{0.1in}
\begin{corollary}
\label{Corollary_multi_normal_bound}
Let $X_1, X_2, \ldots, X_n$ be i.i.d. random variables that follow the N$(\mu, \sigma^2)$ distribution. We have that $\boldsymbol{\hat{\theta}_n(X)} = \left(\bar{X}, \frac{1}{n}\sum_{i=1}^{n}(X_i-\bar{X})^2\right)^{\intercal}$ and the interest is to test $H_0: \mu=0$ against the general alternative. For $h \in \C_b^2(\mathbb{R})$ and $K \sim \chi^2_{1}$, it holds that
\begin{align}
\label{finalboundNormal}
\nonumber \left|\mathbb{E}\left[h\left(-2\log \Lambda\right)\right]- \mathbb{E}[h(K)]\right| & \leq \frac{47,456 \,
\sigma^2\left(\|h''\| + \|h'\|\right)}{\sqrt{n\pi}}
 \max\{1, \sigma^{-9}\}  \\
& + 418,433,114 \, \frac{\|h'\|}{\sqrt{n}} \max\{ 1, \sigma^4\} + 8  \frac{\|h\|}{n}\left(4 + \frac{1}{\sigma^2}\right) . 
\end{align}
\end{corollary}
\vspace{0.1in}
\begin{remark}
\textbf{(1)} For fixed $\sigma^2$, the upper bound in Corollary \ref{Corollary_multi_normal_bound} is of order $\frac{1}{\sqrt{n}}$. There is no claim that the constants are optimal. \\
\textbf{(2)} The normal bound is only small when $\sigma^2$ is neither too large nor too small, so that $n^{-1/7} \ll \sigma^2 \ll n^{1/4}.$
\end{remark}
\begin{proof}
We will use the result of Theorem \ref{Theorem_i.n.i.d}. In this case $d=2$ and $r=1$.
The expected Fisher Information matrix for one random variable is
\begin{equation}
\label{multi_normal_FISHER}
I(\boldsymbol{\theta_0}) = \begin{pmatrix}
\frac{1}{\sigma^2} & 0\\
0 & \frac{1}{2\sigma^4}
\end{pmatrix},\;\; {\rm so\;\;that}\quad [I(\boldsymbol{\theta_0})]^{-1} = \begin{pmatrix}
\sigma^2 & 0\\
0 & 2\sigma^4
\end{pmatrix}.
\end{equation}
The assumptions (R.C.1)-(R.C.5) and (O1)-(O3) are verified for  $\epsilon(\boldsymbol{\theta_0}) < \infty$ and we have that
\begin{equation}
\nonumber \underset{\boldsymbol{\theta}:\left|\theta_s - \theta_{0,s}\right| < \epsilon}{\sup}\left|\frac{\partial^3}{\partial \theta_1^3}\ell(\boldsymbol{\theta};\boldsymbol{X})\right| = 0 =: M_{111}(\boldsymbol{X})
\end{equation}
as well as
\begin{align}
\label{supremumschapter41}
\nonumber & \underset{\boldsymbol{\theta}:\left|\theta_s - \theta_{0,s}\right| < \epsilon}{\sup}\left|\frac{\partial^3}{\partial \theta_2^3}\ell(\boldsymbol{\theta};\boldsymbol{X})\right| = \underset{\boldsymbol{\theta}:\left|\theta_s - \theta_{0,s}\right| 
< \epsilon}{\sup}\left|-\frac{n}{\theta_2^3} + \frac{3}{\theta_2^4}\sum_{i=1}^{n}(X_i - \theta_1)^2\right|\\
& < \frac{n}{(\sigma^2 - \epsilon)^3} + \frac{9n}{(\sigma^2 - \epsilon)^4}\left(\hat{\sigma^2} + \left(\bar{X} - \mu\right)^2 + \epsilon^2\right) =: M_{222}(\boldsymbol{X}).
\end{align}
Moreover,
\begin{align}
\label{supremumschapter42}
\nonumber & \underset{\boldsymbol{\theta}:\left|\theta_s - \theta_{0,s}\right| < \epsilon}{\sup}\left|\frac{\partial^3}{\partial \theta_1 \partial \theta_2^2}\ell(\boldsymbol{\theta};\boldsymbol{X})\right| = \underset{\boldsymbol{\theta}:\left|\theta_s - \theta_{0,s}\right| < \epsilon}{\sup}\left|\frac{\partial^3}{\partial \theta_2^2 \partial \theta_1}\ell(\boldsymbol{\theta};\boldsymbol{X})\right|\\
& < \frac{2n}{(\sigma^2 - \epsilon)^3}\left(\left|\bar{X} - \mu\right| +\epsilon\right) =: M_{122}(\boldsymbol{X})
\end{align}
and
\begin{align}
\label{supremumschapter43}
\nonumber & \underset{\boldsymbol{\theta}:\left|\theta_s - \theta_{0,s}\right| < \epsilon}{\sup}\left|\frac{\partial^3}{\partial \theta_1^2 \partial \theta_2}\ell(\boldsymbol{\theta};\boldsymbol{X})\right| = \underset{\boldsymbol{\theta}:\left|\theta_s - \theta_{0,s}\right| < \epsilon}{\sup}\left|\frac{\partial^3}{\partial \theta_2 \partial \theta_1^2}\ell(\boldsymbol{\theta};\boldsymbol{X})\right|\\
& = \underset{\boldsymbol{\theta}:\left|\theta_s - \theta_{0,s}\right| < \epsilon}{\sup}\left|\frac{n}{\theta_2^2}\right| < \frac{n}{(\sigma^2 - \epsilon)^2} =: M_{112}(\boldsymbol{X}).
\end{align}

We start with the calculation of the first term of the bound in \eqref{final_bound_regression}. From the definition of $U$, $D$, we have that for $c$ as in \eqref{c},
\begin{align}
\nonumber c = \max\left\lbrace \sigma^2, 0 \right\rbrace = \sigma^2.
\end{align}
In addition $Y_{1,i} = \frac{1}{\sigma^2}(X_i - \mu)$, $Y_{2,i} = -\frac{1}{2\sigma^2} + \frac{1}{2\sigma^4}(X_i - \mu)^2$, and
\begin{align}
\nonumber W_1 = \xi = \frac{1}{\sqrt{n}}\sum_{i=1}^{n}Y_{1,i}, \quad W_2 = \eta = \frac{1}{\sqrt{n}}\sum_{i=1}^{n}Y_{2,i}.
\end{align}
We have that $\mathbb{E}\left(W_1^2\right) = \frac{1}{\sigma^2}$ and $\mathbb{E}\left(W_2^2\right) = \frac{1}{2\sigma^4}$. Due to the fact that $R((\xi,\eta), U, D)$, as defined in \eqref{RWLD}, is the minimum of a quantity over $s \in \left\lbrace 1,2\right\rbrace$, then it is upper bounded from the quantity obtained when $s=1$. For ease of understanding and presentation, let $R((\xi,\eta), U, D)_{1}^{(j,k,l)}$ be the value of $R((\xi,\eta), U, D)$, when $s=1$ and for any value of $j,k,l \in \left\lbrace 1,2\right\rbrace$. Then,
\begin{align}
\label{R_WLD_bound}
\nonumber & R((\xi,\eta), U, D) \leq \sum_{j,k,l=1}^2R((\xi,\eta), U, D)_1^{(j,k,l)}\\
& = \sum_{l=1}^{2}\left\lbrace R((\xi,\eta), U, D)_{1}^{(1,1,l)} + 2R((\xi,\eta), U, D)_{1}^{(2,1,l)} + R((\xi,\eta), U, D)_{1}^{(2,2,l)}\right\rbrace
\end{align}
and we proceed with the calculation of each of these quantities. Firstly, for $Z_1$ and $Z_2$ being two independent standard normal random variables,
\begin{align}
\nonumber & \mathbb{E}\left|\left(\left[I(\boldsymbol{\theta_0})\right]^{-1/2}\boldsymbol{Z}\right)_{1}\right| = \sigma\mathbb{E}\left|Z_1\right| = \sigma\sqrt{\frac{2}{\pi}}; \quad 
 \mathbb{E}\left|\left(\left[I(\boldsymbol{\theta_0})\right]^{-1/2}\boldsymbol{Z}\right)_{1}Z_1^2\right|
= 2\sigma\sqrt{\frac{2}{\pi}};  \\
\nonumber & \mathbb{E}\left|\left(\left[I(\boldsymbol{\theta_0})\right]^{-1/2}\boldsymbol{Z}\right)_{1}Z_2^2\right| 
= \sigma^3\sqrt{\frac{2}{\pi}}.
\end{align}
The Cauchy-Schwarz inequality and simple calculations lead to
\begin{align*}
\nonumber & \mathbb{E}\left|Y_{i,1}\right| = \frac{\sqrt{2}}{\sigma\sqrt{\pi}}, \quad \mathbb{E}\left|Y_{i,2}\right| \leq \frac{1}{\sqrt{2}\sigma^2}, \quad \mathbb{E}\left(Y_{i,1}^2\right)= \frac{1}{\sigma^2}, \quad \mathbb{E}\left(Y_{i,2}^2\right) = \frac{1}{2\sigma^4},\\
\nonumber & \mathbb{E}\left(Y_{i,1}Y_{i,2}\right) = 0, \quad \mathbb{E}\left|Y_{i,1}^3\right| = \frac{2\sqrt{2}}{\sigma^3\sqrt{\pi}}, \quad \mathbb{E}\left|Y_{i,2}^3\right|\leq \frac{\sqrt{1510}}{4\sigma^6}, \quad \mathbb{E}\left|Y_{i,1}^2Y_{i,2}\right| \leq \frac{\sqrt{3}}{\sqrt{2}\sigma^4},\\
\nonumber & \mathbb{E}\left|Y_{i,1}Y_{i,2}^2\right| \leq \frac{\sqrt{15}}{2\sigma^5}, \quad \mathbb{E}\left|Y_{i,1}^5\right| = \frac{8\sqrt{2}}{\sigma^5\sqrt{\pi}}, \quad \mathbb{E}\left|Y_{i,2}^5\right| \leq \frac{3\sqrt{2688194}}{8\sigma^{10}}, \quad \mathbb{E}\left|Y_{i,1}^3Y_{i,2}^2\right| \leq \frac{15}{2\sigma^7},\\
& \mathbb{E}\left|Y_{i,1}^2Y_{i,2}^3\right| \leq \frac{\sqrt{4530}}{4\sigma^8}, \quad \mathbb{E}\left|Y_{i,1}^4Y_{i,2}\right|\leq \frac{\sqrt{105}}{\sqrt{2}\sigma^6}, \quad \mathbb{E}\left|Y_{i,1}Y_{i,2}^4\right| \leq \frac{\sqrt{74417}}{4\sigma^9}.
\end{align*}
From \eqref{R_WLD_bound}, we conclude that
\begin{align}
\label{final_bound_R_WLD}
\nonumber  R((\xi,\eta), U, D) & < \frac{\sigma^4}{\sqrt{\pi}}\left\lbrace \vphantom{(\left(\sup_{\theta:|\theta-\theta_0|\leq\epsilon}\left|l^{(3)}(\theta;\boldsymbol{X})\right|\right)^2}\frac{1}{\sigma^9}\left(\frac{57408}{n}+504\right) + \frac{1}{\sigma^8}\left(\frac{18694}{n}+336\right)\right.\\
\nonumber & \;\; \left. + \frac{1}{\sigma^7}\left(\frac{7463}{n} + 1295\right) + \frac{1}{\sigma^6}\left(\frac{3709}{n} + 898\right) + \frac{1}{\sigma^5}\left(\frac{1242}{n} + 893\right)\right.\\
&\;\; \left. + \frac{1}{\sigma^4}\left(\frac{1156}{n} + 776\right) + \frac{428}{\sigma^3} + \frac{385}{\sigma^2} + \frac{105}{\sigma} + 109\vphantom{(\left(\sup_{\theta:|\theta-\theta_0|\leq\epsilon}\left|l^{(3)}(\theta;\boldsymbol{X})\right|\right)^2}\right\rbrace,
\end{align}
which is then used to obtain an upper bound for the first term in the general expression of \eqref{final_bound_regression}. We now proceed to bound $K_1(\boldsymbol{\theta_0})$ in \eqref{K1}. In regards to the first quantity, the expressions for the partial derivatives of the log-likelihood and the fact that in the case of i.i.d. random variables from the normal distribution, $\bar{X}$ and $\hat{\sigma^2}$ are independent random variables \cite[p.218]{Casella}, lead to
\begin{align}
\label{first_term_K1}
\nonumber & 3n\sum_{j=1}^2\sum_{k=1}^2\left[\mathbb{E}\left(Q_j^2Q_k^2\right)\right]^{\frac{1}{2}}\left[{\rm Var}\left(\frac{\partial^2}{\partial\theta_j\partial\theta_k}\log f(X_1|\boldsymbol{\theta_0})\right)\right]^{\frac{1}{2}}\\
\nonumber & = 3n\sqrt{\mathbb{E}\left(\hat{\sigma^2} - \sigma^2\right)^4}\sqrt{{\rm Var}\left(\frac{1}{\sigma^6}(X_1 - \mu)^2\right)}\\
\nonumber & \quad + 6n\sqrt{\mathbb{E}\left(\bar{X}-\mu\right)^2\mathbb{E}\left(\hat{\sigma^2}-\sigma^2\right)^2}\sqrt{{\rm Var}\left(\frac{1}{\sigma^4}\left(X_1 - \mu\right)\right)}\\
& < 3n\sqrt{\frac{16\sigma^8}{n^2}}\sqrt{\frac{2}{\sigma^8}} + 6n\sqrt{2\frac{\sigma^6}{n^2}}\sqrt{\frac{1}{\sigma^6}} = 18\sqrt{2}.
\end{align}
For the second quantity in \eqref{K1}, 
let  $G_{\kappa}\sim\chi^2_{\kappa}$; then 
\begin{align}
\label{bounds_on_Q1_Q2}
\nonumber & \mathbb{E}\left(Q_1^6\right) = \mathbb{E}\left(\bar{X} - \mu\right)^6 = 15\frac{\sigma^6}{n^3}\\
\nonumber & \mathbb{E}\left(Q_2^6\right) = \mathbb{E}\left(\hat{\sigma^2} - \sigma^2\right)^6 = \frac{\sigma^{12}}{n^6}\mathbb{E}\left(G_{n-1} - n\right)^6\\
& \qquad\quad = \frac{\sigma^{12}}{n^3}\left(120 + \frac{940}{n} - \frac{114}{n^2} - \frac{945}{n^3}\right) < \frac{1060}{n^3}\sigma^{12}.
\end{align}
Using now \eqref{multi_normal_FISHER},
\begin{align*}
\nonumber & \mathbb{E}\left(T_{11}^6\right) = 0, \qquad \mathbb{E}\left(T_{12}^6\right) = \mathbb{E}\left(T_{12}^6\right) = \frac{1}{\sigma^{18}}\mathbb{E}\left(\sum_{i=1}^n\left(\frac{X_i - \mu}{\sigma}\right)\right)^6 = \frac{15n^3}{\sigma^{18}},\\
& \mathbb{E}\left(T_{22}^6\right) 
= \frac{1}{\sigma^{24}}\mathbb{E}\left(G_n-n\right)^6 = \frac{40n^3}{\sigma^{24}}\left(3 + \frac{52}{n} + \frac{96}{n^2}\right) \leq \frac{6040n^3}{\sigma^{24}}.
\end{align*}
With inequalities \eqref{multi_normal_FISHER} and  \eqref{bounds_on_Q1_Q2},this yields 
\begin{align}
\label{third_termK1}
 & \sum_{l=1}^{2}\sum_{m=1}^{2}\left[\left[I(\boldsymbol{\theta_0})\right]^{-1}\right]_{lm}\sum_{j=1}^{2}\sum_{k=1}^{2}\sqrt{{\rm Var}\left(\frac{\partial^2}{\partial\theta_l\partial\theta_j}\log f(X_1|\boldsymbol{\theta_0})\right)}\left[\mathbb{E}\left(Q_j^6\right)\mathbb{E}\left(Q_k^6\right)\mathbb{E}\left(T_{mk}^6\right)\right]^{\frac{1}{6}}\\
& < \frac{212}{\sqrt{n}}.\nonumber
\end{align}
Combining the results in \eqref{first_term_K1} and \eqref{third_termK1}, 
\begin{align}
\label{boundK1}
K_1(\boldsymbol{\theta_0}) < 26 + \frac{212}{\sqrt{n}}.
\end{align}
Following the same steps as in \eqref{first_term_K1} and \eqref{third_termK1} under the null hypothesis  $\mu=0$, with  $\hat{\theta}_*(\boldsymbol{X})_1 = \frac{1}{n}\sum_{i=1}^nX_i^2$, 
\begin{align}
\label{bound_K1^*}
K_1^*(\boldsymbol{\theta_0}) < 33 + \frac{220}{\sqrt{n}}.
\end{align}
We proceed to find a bound for $K_2(\boldsymbol{\theta_0})$, as defined in \eqref{K2}. The calculation of the first term is straightforward and\begin{equation}
\label{first_termK2}
2\sqrt{n}\frac{\|h\|}{\epsilon^2}\sum_{j=1}^2\mathbb{E}\left(Q_j^2\right) = 2\sqrt{n}\frac{\|h\|}{\epsilon^2}\left(\frac{\sigma^2}{n}+\frac{\sigma^4}{n}\left(2-\frac{1}{n}\right)\right) < \frac{2\|h\|\sigma^2}{\sqrt{n}\epsilon^2}\left(1+2\sigma^2\right).
\end{equation}
Using \eqref{supremumschapter41}, \eqref{supremumschapter42}, and \eqref{supremumschapter43}, we are able to find an upper bound for the second term in \eqref{K2}. Simple calculations yield
\begin{align}
\label{second_termK2}
& \sqrt{n}\|h'\|\frac{7}{3}\sum_{j=1}^2\sum_{k=1}^2\sum_{s=1}^2\left[\mathbb{E}\left(Q_j^2Q_k^2Q_s^2\right)\right]^{\frac{1}{2}}\left[\mathbb{E}\left[\left(M_{jkm}(\boldsymbol{X})\right)^2\middle|\left|Q_{(m)}\right|<\epsilon\right]\right]^{\frac{1}{2}}\\
\nonumber & = \|h'\|\left\lbrace\vphantom{(\left(\sup_{\theta:|\theta-\theta_0|\leq\epsilon}\left|l^{(3)}(\theta;\boldsymbol{X})\right|\right)^2}\frac{7\sqrt{6}\sigma^4}{\left(\sigma^2 - \epsilon\right)^2} + \frac{56\sqrt{2}\sigma^5}{\left(\sigma^2 - \epsilon\right)^3}\sqrt{\frac{\sigma^2}{n}+\epsilon^2}\right.\\
&\left. \qquad + \frac{7\sqrt{2120}\sigma^6}{3\left(\sigma^2 - \epsilon\right)^3}\sqrt{1 + \frac{162}{\left(\sigma^2 - \epsilon\right)^2}\left((\epsilon + \epsilon^2 +\sigma^2)^2 + \frac{3\sigma^4}{n^2}\right)}\vphantom{(\left(\sup_{\theta:|\theta-\theta_0|\leq\epsilon}\left|l^{(3)}(\theta;\boldsymbol{X})\right|\right)^2}\right\rbrace. \nonumber 
\end{align}
The third term of \eqref{K2} requires the calculation of conditional expectations related to $T_{kj}$ of \eqref{cm} and $M_{qml}(\boldsymbol{X})$, where $k,j,q,m,l \in \left\lbrace 1,2 \right\rbrace$. It is easy to see that both $T_{12}$ and $T_{22}$ can be written as continuous, increasing functions of $Q_1$ and $Q_2$. Therefore, with $Q_{(m)}$ as in \eqref{cm} and for $G_{n} \sim \chi^2_n$, employing Lemma 4.1 of \cite{anastasiou2015assessing}, leads to
\begin{align}
\label{boundT_ij4}
& \mathbb{E}\left(T_{11}^4\middle|\left|Q_{(m)}\right|<\epsilon\right) = 0\\
\nonumber & \mathbb{E}\left(T_{12}^4\middle|\left|Q_{(m)}\right| <\epsilon\right) 
= \mathbb{E}\left(\frac{1}{\sigma^{16}}\left(\sum_{i=1}^n(X_i - \mu)\right)^4\middle|\left|Q_{(m)}\right|<\epsilon\right)\\
\nonumber & 
\qquad\qquad\qquad\qquad\, 
\leq \frac{n^4}{\sigma^{16}}\mathbb{E}\left(\bar{X} - \mu\right)^4 = \frac{3n^2}{\sigma^{12}}\\
\nonumber & \mathbb{E}\left(T_{22}^4\middle|\left|Q_{(m)}\right| <\epsilon\right) = \mathbb{E}\left(\left(\frac{n}{\sigma^4} - \frac{1}{\sigma^2}\sum_{i=1}^n(X_i - \mu)^2\right)^4\middle|\left|Q_{(m)}\right| <\epsilon\right)\\
\nonumber & \qquad\qquad\qquad\qquad\, \leq \mathbb{E}\left(\frac{n}{\sigma^4} - \frac{1}{\sigma^2}\sum_{i=1}^n(X_i - \mu)^2\right)^4\\
& \qquad\qquad\qquad\qquad\, = \frac{1}{\sigma^{16}}\mathbb{E}\left(G_{n} - n\right)^{4} = \frac{12n^2}{\sigma^{16}}\left(1+\frac{4}{n}\right) \leq \frac{60n^2}{\sigma^{16}}. \nonumber 
\end{align}
Using the results of \eqref{supremumschapter41}, \eqref{supremumschapter42}, and \eqref{boundT_ij4}, then after simple calculations, we get for the third term of \eqref{K2} that
\begin{align}
\label{third_termK2}
 & \frac{\|h'\|}{\sqrt{n}}\sum_{q=1}^{2}\sum_{k=1}^{2}\left|\left[\left[I(\boldsymbol{\theta_0})\right]^{-1}\right]_{kq}\right|\sum_{j=1}^{2}\sum_{l=1}^{2}\sum_{s=1}^{2}\left[\mathbb{E}\left(Q_j^2Q_l^2Q_s^2\right)\right]^{\frac{1}{2}}\left[\mathbb{E}\left(T_{kj}^4\middle|\left|Q_{(m)}\right|<\epsilon\right)\right]^{\frac{1}4}\\
\nonumber & \qquad\qquad\qquad\qquad\qquad\qquad\qquad\qquad\qquad \times \left[\mathbb{E}\left(M_{qml}^4(\boldsymbol{X})\middle|\left|Q_{(m)}\right|<\epsilon\right)\right]^{\frac{1}{4}}\\
\nonumber & <  \frac{\|h'\|}{\sqrt{n}}\left(\frac{35\sigma^4}{\left(\sigma^2 - \epsilon\right)^2} + \frac{339\sigma^6}{\left(\sigma^2 - \epsilon\right)^3}\left(\frac{3}{n^2}+\left(\frac{\epsilon}{\sigma}\right)^4\right)^{\frac{1}{4}}\right)\\
& \qquad +\frac{2700\sigma^6\|h'\|}{\sqrt{n}\left(\sigma^2 - \epsilon\right)^3}\left(1+\frac{52488}{\left(\sigma^2 - \epsilon\right)^{4}}\left(\left(\epsilon+\epsilon^2+\sigma^2\right)^4+\frac{105\sigma^8}{n^4}\right)\right)^{\frac{1}{4}} . \nonumber
\end{align}
To find an upper bound for the fourth term of \eqref{K2}, using \eqref{multi_normal_FISHER}, \eqref{supremumschapter41}, \eqref{supremumschapter42}, 
\begin{align} \label{fourth_termK2}
& \frac{\|h'\|}{4\sqrt{n}}\sum_{b=1}^{2}\sum_{k=1}^{2}\sum_{s=1}^{2}\sum_{q=1}^{2}\sum_{l=1}^{2}\sum_{j=1}^{2}\left|\left[\left[I(\boldsymbol{\theta_0})\right]^{-1}\right]_{qb}\right|\sqrt{\mathbb{E}\left(Q_k^2Q_s^2Q_j^2Q_l^2\right)}\\
\nonumber & \qquad\quad \times \left[\mathbb{E}\left(\left(M_{bmk}(\boldsymbol{X})\right)^4\middle|\left|Q_{(m)}\right|<\epsilon\right)\right]^{\frac{1}{4}}\left[\mathbb{E}\left(\left(M_{qjl}(\boldsymbol{X})\right)^4\middle|\left|Q_{(m)}\right|<\epsilon\right)\right]^{\frac{1}{4}}\\
\nonumber & < \frac{\|h'\|}{\sqrt{n}}\left\lbrace \frac{13\sigma^8}{\left(\sigma^2 - \epsilon\right)^4} + \frac{147\sigma^{10}}{\left(\sigma^2 - \epsilon\right)^{5}}\left(\frac{3}{n^2}+\left(\frac{\epsilon}{\sigma}\right)^4\right)^{\frac{1}{4}} +\frac{1361\sigma^{12}}{\left(\sigma^2 - \epsilon\right)^6}\sqrt{\frac{3}{n^2}+\left(\frac{\epsilon}{\sigma}\right)^4}\right.\\
\nonumber & \left.\qquad\quad + \left(1 + \frac{52488}{\left(\sigma^2 - \epsilon\right)^{4}}\left(\left(\epsilon+\epsilon^2+\sigma^2\right)^4+\frac{105\sigma^8}{n^4}\right)\right)^{\frac{1}{4}}\right.\\
\nonumber & \left. \qquad\quad\quad \times \left(\frac{12\sigma^{10}}{\left(\sigma^2 - \epsilon\right)^5} + \frac{369\sigma^{12}}{\left(\sigma^2-\epsilon\right)^6}\left(\frac{3}{n^2}+\left(\frac{\epsilon}{\sigma}\right)^4\right)^{\frac{1}{4}}\right)\right.\\
& \left. \qquad\quad + \frac{\sqrt{362096}\sigma^{12}}{\left(\sigma^2 - \epsilon\right)^6}\sqrt{1 + \frac{52488}{\left(\sigma^2 - \epsilon\right)^{4}}\left(\left(\epsilon+\epsilon^2+\sigma^2\right)^4+\frac{105\sigma^8}{n^4}\right)} .\right\rbrace \nonumber 
\end{align}
The bounds in \eqref{first_termK2}, \eqref{second_termK2}, \eqref{third_termK2}, and \eqref{fourth_termK2} depend on the constant $\epsilon$ as defined in the statement of Theorem \ref{Theorem_i.n.i.d}. For the choice of $\epsilon$, \eqref{supremumschapter41}, \eqref{supremumschapter42} and \eqref{supremumschapter43} require that $0< \epsilon<\sigma^2$. There is trade off related to the choice of $\epsilon$ between the expressions \eqref{first_termK2}, and \eqref{fourth_termK2}. We choose $\epsilon = \frac{\sigma^2}{2}$. Using this value in \eqref{first_termK2}, \eqref{second_termK2}, \eqref{third_termK2}, and \eqref{fourth_termK2}, leads to
\begin{align}
\label{boundK2}
 & K_2(\boldsymbol{\theta_0}) < \frac{8\|h\|}{\sqrt{n}\sigma^2}\left(1+2\sigma^2\right) + \|h'\|\left\lbrace\vphantom{(\left(\sup_{\theta:|\theta-\theta_0|\leq\epsilon}\left|l^{(3)}(\theta;\boldsymbol{X})\right|\right)^2} 28\sqrt{6} + 448\sqrt{\frac{2}{n} +\frac{\sigma^2}{2}} + \frac{348}{\sqrt{n}}\right.\\
\nonumber & \left.\;+ 860\sqrt{1 + 648\left(\left(\frac{3}{2}+\frac{\sigma^2}{4}\right)^2+\frac{3}{n^2}\right)} + \frac{7416}{\sqrt{n}}\left(\frac{3}{n^2}+\frac{\sigma^4}{16}\right)^{\frac{1}{4}}\right.\\
\nonumber & \left. \; + \frac{1}{\sqrt{n}}\left(1+839808\left(\left(\frac{3}{2}+\frac{\sigma^2}{4}\right)^4+\frac{105}{n^4}\right)\right)^{\frac{1}{4}}\left(21984 + 23616\left(\frac{3}{n^2}+\frac{\sigma^4}{16}\right)^{\frac{1}{4}}\right)\right.\\
& \left. \; + \frac{87104}{\sqrt{n}}\sqrt{\frac{3}{n^2} + \frac{\sigma^4}{16}} + \frac{38512}{\sqrt{n}}\sqrt{1 + 839808\left(\left(\frac{3}{2}+\frac{\sigma^2}{4}\right)^4+\frac{105}{n^4}\right)}\vphantom{(\left(\sup_{\theta:|\theta-\theta_0|\leq\epsilon}\left|l^{(3)}(\theta;\boldsymbol{X})\right|\right)^2}\right\rbrace. \nonumber
\end{align}
It remains to find an upper bound for $K_2^*(\boldsymbol{\theta_0})$, which is the version of $K_2(\boldsymbol{\theta_0})$ under the null hypothesis of $\mu = 0$. This requires the calculation of conditional expectations related to $T^*_{11}$ of \eqref{cm} and $M^*_{111}(\boldsymbol{X})$ as defined in (O1). Simple calculations yield
\begin{align}
\nonumber & M^*_{111}(\boldsymbol{X}) = \frac{n}{\left(\sigma^2 - \epsilon\right)^3} + \frac{3}{\left(\sigma^2 - \epsilon\right)^4}\sum_{i=1}^nX_i^2, \qquad T^*_{11} = \frac{n}{\sigma^4} - \frac{1}{\sigma^6}\sum_{i=1}^nX_i^2.
\end{align}
Using the above results and Lemma 4.1 from \cite{anastasiou2015assessing}, we obtain that for $\epsilon = \frac{\sigma^2}{2}$,
\begin{align}
\label{boundK2*}
 & K_2^*(\boldsymbol{\theta_0}) = 2\sqrt{n}\frac{\|h\|}{\epsilon^2}\mathbb{E}\left(Q^*_1\right)^2 + \sqrt{n}\|h'\|\frac{7}{3}\sqrt{\mathbb{E}\left(Q^*_1\right)^6}\sqrt{\mathbb{E}\left[\left(M^*_{111}(\boldsymbol{X})\right)^2\middle|\left|Q^*_{(m)}\right|<\epsilon\right]}\\
\nonumber & \;\; + \frac{2\|h'\|\sigma^4}{\sqrt{n}}\sqrt{\mathbb{E}\left(Q^*_j\right)^6}\left[\mathbb{E}\left(\left(T^*_{11}\right)^4\middle|\left|Q^*_{(m)}\right|<\epsilon\right)\right]^{\frac{1}{4}}\left[\mathbb{E}\left(\left(M^*_{111}(\boldsymbol{X})\right)^4\middle|\left|Q^*_{(m)}\right|<\epsilon\right)\right]^{\frac{1}{4}}\\
\nonumber & \;\; + \frac{\sigma^4\|h'\|}{2\sqrt{n}}\sqrt{\mathbb{E}\left(Q^*_1\right)^8}\sqrt{\mathbb{E}\left(\left(M^*_{111}(\boldsymbol{X})\right)^4\middle|\left|Q^*_{(m)}\right|<\epsilon\right)}\\
& < \frac{16\|h\|}{\sqrt{n}} + 15958\|h'\| + \frac{13527046}{\sqrt{n}}\|h'\| . \nonumber
\end{align}
Applying the results of \eqref{final_bound_R_WLD}, \eqref{boundK1}, \eqref{bound_K1^*}, \eqref{boundK2}, and \eqref{boundK2*}, to the expression of the general upper bound in \eqref{final_bound_regression} yields the assertion of the corollary as expressed in \eqref{finalboundNormal}.  
\end{proof}
\subsection{Example: Logistic regression} 
In binomial regression, the data are i.i.d. observations $(\boldsymbol{X_i}, Y_i), i=1, \ldots, n$, where $\boldsymbol{X_i} \in \RR^d$ and $Y_i \in \{0 , 1\}$, see for example \cite{van1998asymptotic}, p.66. The binary regression model is that 
$$\PP_{\boldsymbol{\theta}} (Y_i = 1 | \boldsymbol{X_i} = \boldsymbol{x}) = \psi (\boldsymbol{\theta}^{\intercal}\boldsymbol{x})$$ for $\boldsymbol{\theta} \in \RR^d$ and $\psi: \RR \rightarrow [0,1]$ continuously differentiable, monotone, with derivatives bounded away from 0 and $\infty$. For logistic regression, $$\psi(\theta) = \frac{1}{1+e^{-\theta}}$$
and in this example we restrict ourselves to this case, although generalisations are straightforward. We assume that the distribution of $\boldsymbol{X}$ is such that $\boldsymbol{X}$ has finite moments up to order 6. To ensure that the MLE $\boldsymbol{\hat{\theta}_n(X)}$ for $\boldsymbol{\theta}$ exists and is unique, we assume that the $\boldsymbol{X_i}$'s do not concentrate on a $(d-1)$-dimensional affine subspace of $\RR^d$. 

Consider as in \cite{sur2017likelihood} to test the simple hypothesis that $\boldsymbol{\theta_0} =\boldsymbol{0}$ against the general alternative. The likelihood in this case is 
$$L(\boldsymbol{\theta}; (\boldsymbol{x_i}, y_i), i=1, \ldots, n)= \prod_{i=1}^n \left\{ \psi ( \boldsymbol{\theta}^{\intercal} \boldsymbol{x_i})^{y_i} ( 1 - \psi (\boldsymbol{\theta}^{\intercal} \boldsymbol{x_i}))^{1-y_i}\right\} $$
so that the score function is 
$$S(\boldsymbol{\theta}) = \frac{y - \psi (\boldsymbol{\theta}^{\intercal} \boldsymbol{x})}{\psi (\boldsymbol{\theta}^{\intercal} \boldsymbol{x}) ( 1- \psi (\boldsymbol{\theta}^{\intercal} \boldsymbol{x})) } \psi' (\boldsymbol{\theta}^{\intercal}\boldsymbol{x})\boldsymbol{x}$$
while the Fisher information matrix is $I(\boldsymbol{\theta}) = \EE \left[  \psi'(\boldsymbol{\theta}^{\intercal} \boldsymbol{X})) \boldsymbol{X} \boldsymbol{X}^{\intercal} \right].$
For testing $H_0: \theta_1 = 0$ we have for $\boldsymbol{x} = (x_1, \ldots, x_d)$ 
$$\xi (\boldsymbol{x}) = \frac{1}{\sqrt{n}} ( y_1 - \psi (\boldsymbol{\theta}^{\intercal} \boldsymbol{x})) x_1$$
and 
$$ \boldsymbol{\eta} (\boldsymbol{x}) = \frac{1}{\sqrt{n}}\left( ( y_2 - \psi (\boldsymbol{\theta}^{\intercal} \boldsymbol{x}) )x_2, 
\ldots, ( y_p - \psi (\boldsymbol{\theta}^{\intercal} \boldsymbol{x}) )x_p \right)^{\intercal}.$$

To check the assumptions on the third derivative of the log-likelihood we calculate 
\begin{equation} 
\nonumber \frac{\partial^3}{\partial \theta_k \partial \theta_j \partial \theta_i} \ell (\boldsymbol{\theta}; (\boldsymbol{x}, y)) = \frac{e^{\boldsymbol{\theta}^{\intercal} \boldsymbol{x}}(1 + e^{2 \boldsymbol{\theta}^{\intercal} \boldsymbol{x}} - 4 e^{\boldsymbol{\theta}^{\intercal} \boldsymbol{x}})}{ \left(1+ e^{\boldsymbol{\theta}^{\intercal} \boldsymbol{x}}\right)^4}x_i x_j x_k 
\end{equation}
so that 
\begin{equation}
\nonumber \left| \frac{\partial^3}{\partial \theta_k \partial \theta_j \partial \theta_i} \ell (\boldsymbol{\theta}; (\boldsymbol{x}, y)) \right| \le |x_i x_j x_k | =: M_{k,j,i}(\boldsymbol{x}).
\end{equation}
In particular $\EE ( M_{k,j,i}(\bX)^2) = \EE ( X_i^2 X_j^2 X_k^2) $
and  $\EE ( M_{k,j,i}(\bX)^4) = \EE ( X_i^4 X_j^4 X_k^4) $. 

In order to apply Proposition \ref{gauntreinertcor} we use the variables
$$Z_{i,j} =   \frac{\partial}{\partial \theta_j} 
(y_i \log \psi (\boldsymbol{\theta}^{\intercal} \boldsymbol{x_i}) + (1-y_i) \log (1 - \psi (\boldsymbol{\theta}^{\intercal} \boldsymbol{x_i})) = \left( y_i - \psi (\boldsymbol{\theta}^{\intercal} \boldsymbol{x_i}) \right)x_{i,j}$$ 
for $i=1, \ldots, n, j=1, \ldots, d$; they play the role of the $Y_{i,j}$ from Proposition \ref{gauntreinertcor}. 
Due to the binomial structure and the fact that $ |y_i - \psi (\boldsymbol{\theta}^{\intercal} \boldsymbol{X_i}) | \le 1$ it holds that
$\EE | Z_{i,j}| \le \EE |X_{i,j}|. $
We bound the third and fifth moments in \eqref{RWLD}, by defining
\begin{align}
\nonumber & \mu_j^{(3)} = \max_{i_1, i_2, i_3 \in \left\lbrace 1,2,\ldots,d\right\rbrace}\left\lbrace\EE | Y_{i_1,j} Y_{i_2, j}Y_{i_3,j}|, \;\left|\mathbb{E}\left(Y_{i_1,j}Y_{i_2,j}\right)\right|\EE\left|Y_{i_3,j}\right|\right\rbrace\\
\nonumber & \mu_j^{(5)} = \max_{i_1, i_2, \ldots, i_5 \in \left\lbrace 1,2,\ldots,d\right\rbrace}\left\lbrace \EE | Y_{i_1,j} Y_{i_2, j}Y_{i_3,j}Y_{i_4,j}Y_{i_5,j}|,\;\left|\mathbb{E}\left(Y_{i_1,j}Y_{i_2,j}\right)\right|\EE\left|Y_{i_3,j}Y_{i_4,j}Y_{i_5,j}\right|\right\rbrace.
\end{align}
In this example, for $k=3,5$,
$$ \mu_j^{(k)} \le   \max_{i_1, \ldots, i_k \in \{1, \ldots, d\} } \EE \left| \prod_{s=1}^k X_{j, i_s} \right| .$$  
Note that for logistic regression the MLE in general does not have a closed form. Hence we cannot evaluate \eqref{cm} explicitly, although with given data sets a numerical evaluation is possible. For our purposes it suffices to illustrate the applicability of the bound as well as its behaviour in terms of $d$ and $n$. 

From Remark \ref{remark_multi_order}. \textbf{(5)}, the overall order of the bound in the chisquare approximation for the likelihood ratio test is at most $d^7 n^{-\frac12}$ and therefore, the chisquare approximation is justified when $ d = o(n^{\frac{1}{14}})$. This bound is lower than the bound of $d = o(n^{\frac23})$ reported in \cite{portnoy1988asymptotic}, but our bound is explicit and is derived in a more general setting.

\bibliographystyle{Chicago}
\bibliography{WilksJune10_general}

\end{document}